\documentclass[11pt,]{amsart}
\usepackage{graphicx} 
\usepackage[utf8]{inputenc}
\usepackage{amsmath,amssymb}
\usepackage{wrapfig}
\usepackage{url}

\usepackage{mathtools}
\usepackage{graphicx}
\usepackage{stmaryrd}
\usepackage{amsthm}
\usepackage{mathrsfs}

\usepackage{xcolor}
\usepackage[shortlabels]{enumitem}
\usepackage{comment}
\usepackage{relsize}
\usepackage{ dsfont }
\usepackage{hyperref}
\usepackage{stmaryrd}
\usepackage{yhmath}
\usepackage{geometry}
\newtheorem{theorem}{Theorem}
\numberwithin{theorem}{section} 
\numberwithin{equation}{section}
\newtheorem{corollary}[theorem]{Corollary}
\newtheorem{prop}[theorem]{Proposition}
\theoremstyle{definition}
\newtheorem{example}[theorem]{Example}
\newtheorem{lemma}[theorem]{Lemma}
\theoremstyle{definition}
\newtheorem{obs}{Remark}
\theoremstyle{definition}
\newtheorem{definition}[theorem]{Definition}

\newcommand{\T}{\mathbb{T}}
\newcommand{\R}{\mathbb{R}}
\newcommand{\Z}{\mathbb{Z}}
\newcommand{\N}{\mathbb{N}}

\newcommand{\Real}{\text{Re}}
\newcommand{\Imag}{\text{Im}}
\newcommand{\C}{\mathbb{C}}
\newcommand{\Arg}{\operatorname{Arg}}

\DeclareFontFamily{U}{mathx}{\hyphenchar\font45}
\DeclareFontShape{U}{mathx}{m}{n}{
      <5> <6> <7> <8> <9> <10>
      <10.95> <12> <14.4> <17.28> <20.74> <24.88>
      mathx10
      }{}
\DeclareSymbolFont{mathx}{U}{mathx}{m}{n}
\DeclareFontSubstitution{U}{mathx}{m}{n}
\DeclareMathAccent{\widecheck}{0}{mathx}{"71}

\newcommand{\Cinfty}{C^\infty_{\theta,T}(\R^n)}

\newcommand{\D}{\mathcal{D}'_{\theta,T}(\R^n)}

\newcommand{\cinfty}{C^\infty_{\theta,T}(\R^2)}

\renewcommand{\O}{\Omega_{\theta,T}}

\newcommand{\defeq}{\vcentcolon=}
\newcommand{\eqdef}{=\vcentcolon}
\DeclareMathOperator*{\esssup}{ess\,sup}

\allowdisplaybreaks

\title{A Fourier analysis for $(\theta,T)$-periodic functions and applications}

\author[A. Kowacs]{Andr\'e Pedroso Kowacs}
\address{
  Department of Mathematics
  \endgraf
 Universidade de São Paulo (ICMC-USP), Brazil
}
\email{andrekowacs@gmail.com}

\author[M. Ap. Silva]{Marielle Aparecida Silva}
\address{
Universidade Tecnológica Federal do Paraná 
(UTFPR), Brazil
}
\email{mariellesilva@utfpr.edu.br}

\thanks{This study was financed, in part, by the São Paulo Research Foundation (FAPESP), Brasil.
Process Number 2025/08151-5.}
	


\begin{document}
\begin{abstract}
    We develop a Fourier analysis for a generalization of the class of periodic functions, often referred to as $(\theta, T)$-periodic functions, and prove several properties and inequalities related to the Fourier transform, including a type of Poincaré inequality, which extend the periodic case. As an application, we employ this analysis to show that a continuous linear operator acting on smooth $(\theta, T)$-periodic functions is globally hypoelliptic/solvable if and only if the corresponding operator which acts on periodic functions is globally hypoelliptic/solvable, and characterize the global hypoellipticity/solvability of a class of first order differential operators acting on the set of smooth $(\theta, T)$-periodic functions.
\end{abstract}
\maketitle

\section{Introduction}

Fourier analysis is an important tool with a wide range of applications. For instance, it is extensively applied to the study of linear operators acting on periodic functions. In particular, the study of global regularity of (pseudo)differential operators on the torus, which has been a very active area of research on the last few decades (for instance, see \cite{Berga_perturbations,BDPR_existence_regularity,BP_solvability_invlutive,GW_Liouville,GW_hypo,Hounie}), relies deeply on Fourier analysis  and shows connections to Diophantine approximation theory. Recent developments of the Fourier analysis on compact manifolds and, in particular, compact Lie groups, have motivated similar studies in these more general settings (for instance, see \cite{AFR_solvability_sums_squares,KW_strongly_invariant,KWR_vector,RTW_Lie_groups}).

On the other hand, $(\theta, T)$-periodic functions also referred to (among other names) as $(\omega,c)$-periodic, have proved to be relevant in several analytical and applied contexts, especially in ordinary and functional differential equations. See, for instance, \cite{agaoglou2020existence}, \cite{apsilva_marielle} and \cite{wang2019omega}. In these works, the authors establish conditions under which the solutions of certain differential equations are  $(\theta, T)$-periodic, generalizing the classical results of periodicity and, moreover, developing a Floquet theory for this class of equations.
A function $f : \R \to \C$ is said to be $(\theta, T)$-periodic if $f(t + T)= \theta f(t)$, for every $t \in \R$ and $\theta \in \C$. This theory  was introduced and developed in \cite{alvarez2018omega} and \cite{alvarez2019omega}, motivated by systems with phase symmetry, such as the Mathieu’s equation, which admits $(\theta, T)$-periodic solutions.  Clearly, this class of functions generalizes the classical case of periodicity $(\theta =1)$ and also of antiperiodicity $(\theta = -1)$. 

In this work, we investigate $(\theta, T)$-periodic functions in the framework of Fourier analysis and distributions spaces. We establish a Poincaré-type inequality for this class of functions and present global regularity results for first-order differential operators, showing that certain hypoellipticity properties are preserved in the environment of $(\theta, T)$-periodic functions.

It is worth mentioning that the theory developed in this paper can alternatively be seen in the context of nonharmonic analysis, which first appeared in \cite{PW_nonharmonic} and was recently further developed in \cite{nonharmonic_general}. This analysis  was also further applied to the corresponding setting in \cite{wagner_nonharmonic}. However, we note that the approach taken in the present paper seems to provide new and more precise results, and also the equivalent formulation for the Fourier analysis of $(\theta,T)$-periodic functions in the previously mentioned  references only considers real valued $\theta$.

 The present paper is organized as follows.  In Section \ref{Section2}, we deal with smooth $(\theta, T)$-periodic functions defined on $\R^n$ and the corresponding distributions. We present illustrative examples and, by introducing the operator $\Omega_{\theta,T}$, establish fundamental results in the framework of Fourier analysis. See Propositions \ref{prop_Fourier_L1}, \ref{prop_L2} and \ref{prop_Fourier_properties}. Moreover, we discuss Sobolev spaces in the set of 
 $(\theta, T)$-periodic distributions and provide a type of Sobolev embedding theorem in Proposition \ref{prop_Sobolev_embedding}. 
 
 Theorem \ref{theo_poincare} in Section \ref{Section3} presents a version of the Poincaré inequality for $(\theta, T)$-periodic functions, extending the classical result corresponding to the case $\theta =1$.
 
In Section \ref{Section4}, we apply the Fourier analysis developed in Section \ref{Section2} to the study of global hypoellipticity and global solvability of continuous linear operators acting on smooth $(\theta,T)$-periodic functions. As an application of the main result obtained in this section, we  prove that any continuous linear operator $P$ acting on smooth $(\theta,T)$-periodic functions is globally hypoelliptic/solvable over the smooth $(\theta,T)$-periodic functions if and only if the induced $\tilde P$ which acts on $2\pi$-periodic functions is  globally hypoelliptic/solvable. Using this result, we obtain necessary and sufficient conditions for the class of differential operators $\partial_{x_1}+c(x_1)\partial_{x_2}+q(x_1,x_2)$, where $c,q$ are smooth $T$-periodic functions, to be globally hypoelliptic/solvable over the smooth $(\theta,T)$-periodic functions. These results are demonstrated in Theorem \ref{theo_gh_gs_general} and Propositions \ref{prop_GH_GS_cte}, \ref{prop_GH_GS_real} and \ref{prop_GH_GS_complex}.

In Appendix \ref{sec_appendix} we present the proof of some results about global hypoellipticity and global solvability for periodic functions (that is, on the torus), which although can be found elsewhere (or at least derived from other existing results), we chose to include for the sake of completeness. 

Finally, in Appendix \ref{sec_app_solution}, we adapt well known formulas for periodic solutions of first order periodic ordinary differential equations to the $(\theta,T)$-periodic setting. These formulas are used in the proofs in Appendix \ref{sec_appendix} in the periodic setting.

 \section{\texorpdfstring{Analysis on ($\theta,T$)-periodic functions}{Analysis on (θ,T)-periodic functions}}\label{Section2}
 In this section, we present the theory of distributions and Fourier analysis for the space of $(\theta, T)$-periodic functions. For classical theory involving periodic functions, the  reader may want to consult \cite{Grafakos,Katznelson}.

Let $\C$ be the field of complex numbers and denote by $\C^n_{*}$ the set $(\C \backslash {\{0}\})^n$. 
Given $\theta\in \C^n_{*}$ and $T>0$, we say that a measurable function $f:\R^n\to \C$ (with respect to the usual Lebesgue measure on $\R^n)$ is $(\theta,T)$-periodic  if it satisfies 
\begin{equation*}
    f(x+Te_j)=\theta_jf(x), 
    \end{equation*}
for every $x\in \R^n$, $1\leq j\leq n$, where ${e_1,\dots,e_n}$ denotes the canonical basis in $\R^n$.  In this case, we write $f\in L^0_{\theta,T}(\R^n)$.   The set of smooth $(\theta,T)$-periodic functions $C^\infty_{\theta,T}(\R^n)$ is then defined as the set of all smooth functions in $L^0_{\theta,T}(\R^n)$.

{ In what follows, we present some examples.}

\begin{example}
    Consider $f\in C^\infty(\R^n)$ given by 
    \begin{equation*}
        f(x)=e^{\xi\cdot x}=e^{\sum_{i=1}^n\xi_ix_i},
    \end{equation*}
    for some $\xi\in\R^n$ and every $x\in\R^n$. Then $f\in C^\infty_{\theta(T),T}(\R^n)$ for every $T>0$, where 
    \begin{equation*}
        \theta(T)= (e^{\xi_1T},\dots,e^{\xi_nT}).
    \end{equation*}
\end{example}

\begin{example}
    Given $\theta \in \C_{*}=\C \setminus{\{0}\}$, $T>0$ and $f\in C^\infty(\R)$ {with compact support}, define 
 \begin{equation*}
     {f}_{\theta,T}(x)=\sum_{k\in\Z}\theta^{-k}f(x+kT),
 \end{equation*}
 for every $x\in\R$. Then a simple computation shows that ${f}_{\theta,T}\in C_{\theta,T}^\infty(\R)$.
\end{example}

\begin{example}
    Given {$ \theta \in \C^n_{*}$}, $T>0$ and $f\in \Cinfty$, consider the function $\mu_k f:\R^n\to \C$ given by $g(x)=f(kx)$, where $k\in \Z\backslash\{0\}$. Then 
    \begin{equation*}
        \mu_kf(x+T)=f(k(x+T))=f(kx+kT)=\theta^kf(kx)=\theta^k(\mu_kf)(x)
    \end{equation*}
    and
    \begin{equation*}
        \mu_kf(x+{\textstyle\frac{T}{|k|}})=f(k(x+{\textstyle\frac{T}{|k|}}))=f(kx+\operatorname{sign}(k)T)=\theta^{\operatorname{sign}(k)} f(kx)=\theta^{\operatorname{sign}(k)}(\mu_kf)(x),
    \end{equation*}
    where $\theta^k\defeq (\theta_1^k,\dots, \theta_n^k)$, so $\mu_kf\in C^\infty_{\theta^k,T}(\R^n)\cap C^\infty_{\theta^{\operatorname{sign}(k)},\frac{T}{|k|}}(\R^n)$.
\end{example}

For the rest of this paper, let $\log:\C_*\to \C$ denote a fixed choice of a complex logarithm function; for instance, one such choice could be the principal value of the complex logarithm. It is important to note that the validity of our results remains intact regardless of this choice, as is evident from the proofs and further remarks below.

\begin{definition} We define the mapping $\Omega_{\theta,T}:L^0_{\theta,T}(\R^n)\to L^0(\R^n/(2\pi\Z^n))\cong L^0([0,2\pi]^n)$ given by
\begin{equation*}
    \Omega_{\theta,T} f(x)\vcentcolon=e^{-\sum_{i=1}^nx_j\frac{\log(\theta_j)}{2\pi}}f(\textstyle\frac{T}{2\pi}x)=\vcentcolon e^{-x\cdot\frac{\log(\theta)}{2\pi}}f(\textstyle\frac{T}{2\pi}x),
\end{equation*}
for every $f\in L^0_{\theta,T}(\R^n)$, $x\in\R^n$. 

\end{definition}
Note that $\Omega_{\theta,T}$ is a bijection, with inverse given by 
\begin{equation*}
    \Omega_{\theta,T}^{-1} f(x)\vcentcolon=e^{\sum_{i=1}^nx_j\frac{\log(\theta_j)}{T}}f(\textstyle\frac{2\pi}{T}x)=\vcentcolon e^{x\cdot\frac{\log(\theta)}{T}}f(\textstyle\frac{2\pi}{T}x).
\end{equation*}

Recall that for $1\leq p\leq \infty$, there exists a natural identification $L^p(\T^n)\cong L^p([0,2\pi]^n)$, where we consider the normalized Lebesgue measure $\frac{1}{(2\pi)^n}dx$ on $[0,2\pi]^n$. Also, we will consider the space $L^p([0,T]^n)$ with the normalized Lebesgue measure as well.  With that in mind, for $1\leq p\leq \infty$,  we define the space $L^p_{\theta,T}(\R^n)$ as the set of equivalence classes of all functions $f\in L^0_{\theta,T}(\R^n)$ such that
\begin{equation*}
    \|f\|_{L^p_{\theta,T}(\R^n)}=\|\O f\|_{L^p(\T^n)}=\|\O f\|_{L^p([0,2\pi]^n)}<\infty,
\end{equation*}
where $f\sim g$ if and only if $\|f-g\|_{L^p_{\theta,T}(\R^n)}=0$. We equip this space with the norm $\|\cdot\|_{L^p_{\theta,T}(\R^n)}$ with the usual identification of functions and their equivalence classes of $L^p$ spaces. 

\begin{obs}
    Note that for $1\leq p<\infty$, 
\begin{align*}
    \|f\|_{L^p_{\theta,T}(\R^n)}^p&=\frac{1}{(2\pi)^n}\int_{[0,2\pi]^n}{\textstyle |f(\frac{T}{2\pi}x)}|^p|e^{-x\cdot \frac{\log(\theta)}{2\pi}}|^p dx\\
    &=\int_{[0,2\pi]^n}{\textstyle |f(\frac{T}{2\pi}x)}|^pe^{-p\sum_{j=1}^n x_j \frac{\Real(\log(\theta_j))}{2\pi}} dx
    \\
    &=\int_{[0,2\pi]^n}{\textstyle |f(\frac{T}{2\pi}x)}|^pe^{-p\sum_{j=1}^n x_j \frac{\log(|\theta_j|)}{2\pi}} dx\\
    &=\int_{[0,2\pi]^n}{\textstyle |f(\frac{T}{2\pi}x)}|^pe^{-p\log(\prod_{j=1}^n  |\theta_j|^{\frac{x_j}{2\pi}})} dx\\
    &=\frac{1}{T^n}\int_{[0,T]^n}|f(x)|^p\Big(\prod_{j=1}^n|\theta_j|^{-\frac{x_j}{T}}\Big)^pdx,
\end{align*}
so in fact $L^p_{\theta,T}(\R^n)$ can be viewed as the $L^p$ space on $[0,T]^n$ with respect to the measure
\begin{equation*}
    d\mu \vcentcolon=\frac{1}{T^n} |e^{-x\cdot \frac{\log(\theta)}{2\pi}}|^p dx=\frac{1}{T^n}\prod_{j=1}^n|\theta_j|^{-\frac{px_j}{T}}dx,
\end{equation*}
where $dx$ denotes the usual Lebesgue measure on $[0,T]^n$. 

Similarly, for $p=\infty$, we have that
\begin{align*}
    \|f\|_{L^\infty_{\theta,T}(\R^n)}&=\esssup_{x\in[0,2\pi]^n}{\textstyle|f(\frac{T}{2\pi}x)}||e^{-x\cdot \frac{\log(\theta)}{2\pi}}|\\
    &=\esssup_{x\in [0,T]^n}{\textstyle|f(x)}|\prod_{j=1}^n|\theta_j|^{-\frac{x_j}{T}}\\
    &\asymp\esssup_{x\in [0,T]^n}{\textstyle|f(x)}|=\|f\|_{L^\infty([0,T]^n)}.
\end{align*}
Therefore for every $1\leq p\leq \infty$, $L^p_{\theta,T}(\R^n)$ is a Banach space, which by definition is isometrically isomorphic to $L^p(\T^n)$ via the restriction $\O: L^p_{\theta,T}(\R^n)\to L^p(\T^n)$, and in particular,  $L^p_{\theta,T}(\R^n)=L^p([0,T]^n)$ when $|\theta_j|=1$, for every $j=1,\dots,n$.
\end{obs}

Next, we define a topology on $C^\infty_{\theta,T}(\R^n)$ from the usual Fréchet space topology on $C^\infty(\T^n)$ via the bijection $\Omega_{\theta,T}$. In particular, since the topology on $C^\infty(\T^n)$ is the Fréchet topology induced by the countable family of (semi)norms
\begin{equation*}
    \tilde p_N(f)\defeq \sum_{|\alpha|\leq N}\sup_{x\in [0,2\pi]^n}|\partial_x^\alpha f(x)|,
\end{equation*}
for $N\in\N_0$ and $f\in C^\infty(\T^n)$, it follows that the topology on $\Cinfty$ is induced by the countable family of (semi)norms
\begin{align*}
    \tilde p_N(\O f)&= \sum_{|\alpha|\leq N}\sup_{x\in [0,2\pi]^n}|\partial_x^\alpha[\O f(x)]|\\
    &=\sum_{|\alpha|\leq N}\sup_{x\in [0,2\pi]^n}|\partial_x^\alpha[e^{-\sum_{i=1}^n\log(\theta_i)\frac{x_i}{2\pi}}f(\textstyle\frac{T}{2\pi}x)]|\\
    &\asymp\sum_{|\alpha|\leq N}\sup_{x\in [0,T]^n}|\partial_x^\alpha f(x)|\eqdef p_N(f),
\end{align*}
where $N\in\N_0$, $f\in\Cinfty$. This also makes $C^\infty_{\theta,T}(\R^n)$ into a Fréchet space and we have that $f_j\to f$ in $C^\infty_{\theta,T}(\R^n)$ if and only if $\Omega_{\theta,T}f_j\to \Omega_{\theta,T}f$ in $C^\infty(\T^n)$.

 Now, we present the space of distributions in the environment of $(\theta, T)$-periodic functions. Let $\mathcal{D}'_{\theta,T}(\R^n)$ denote the set of $(\theta,T)$-periodic distributions, that is, the set of continuous linear functionals on $C^\infty_{\theta,T}(\R^n)$. Then, since $\Omega_{\theta,T}$ is linear, a linear functional $u:C^\infty_{\theta,T}(\R^n)\to \C$ is in $\mathcal{D}'_{\theta,T}(\R^n)$ if and only if the linear functional $\tilde u :C^\infty(\T^n)\to \C$ given by
\begin{equation*}
   \langle \tilde u,f\rangle =  \langle u,\Omega_{\theta,T}^{-1}f\rangle, 
\end{equation*}
is continuous (and therefore in $\mathcal{D}'(\T^n))$. Moreover, any $\tilde u\in\mathcal{D}'(\T^n)$ defines $u\in \mathcal{D}'_{\theta,T}(\R^n)$ by 
\begin{equation*}
    \langle u,f\rangle = \langle \tilde u,\Omega_{\theta,T} f\rangle.
\end{equation*}
Therefore the transpose $\Omega_{\theta,T}^t:\mathcal{D'}(\T^n)\to \mathcal{D}'_{\theta,T}(\R^n)$ given by 
\begin{equation*}
    \langle \Omega_{\theta,T}^t \tilde u,f\rangle = \langle \tilde u,\Omega_{\theta,T} f \rangle,
\end{equation*}
for every $\tilde u\in\mathcal{D}'(\T^n)$, $f\in C^\infty_{\theta,T}(\R^n)$ establishes a bijection between $\D$ and $\mathcal{D'}(\T^n)$. Its inverse $(\O^{t})^{-1}=(\O^{-1})^t:\D\to\mathcal{D}'(\T^n)$, is given by
\begin{equation*}
    \langle (\O^{t})^{-1}u,f\rangle=\langle u , \O^{-1}f\rangle,
\end{equation*}
for every $u\in\D$, $f\in C^\infty(\T^n)$.

\begin{obs}
  Note that any $f\in \Cinfty$ defines $\tilde u_f\in\mathcal{D}'(\T^n)$ via the usual integration duality formula
\begin{equation*}
    \langle \tilde u_f,g\rangle=  \frac{1}{(2\pi)^n}\int_{[0,2\pi]^n} \O f(x)g(x)dx.
\end{equation*}
Hence, it also defines a distribution $u_f=\O^t\tilde u_f\in\D$ given by
\begin{align*}
    \langle u_f,g\rangle&=\langle \O^t \tilde u_f,g\rangle=\langle \tilde  u_f,\O g\rangle\\
    &=\frac{1}{(2\pi)^n}\int_{[0,2\pi]^n}\O f(x) \O g(x)dx\\
    &=\frac{1}{T^n}\int_{[0,T]^n}f(x)g(x)e^{-2x\cdot \frac{\log(\theta)}{T}}dx.
\end{align*}
\end{obs} 

 By identifying $f\in \Cinfty$  with $u_f\in \D$, and the set of  smooth functions on $\T^n$ as a subspace of the set of distributions,  we have the following result.
 \begin{prop}\label{prop_transpose_identity}
     For every $f\in\Cinfty\subset \D$, we have that
     \begin{equation*}
      (\O^{-1})^tf=\O f \in   C^\infty(\T^n)\subset \mathcal{D}'(\T^n),
     \end{equation*}
     that is, 
     \begin{equation*}
         (\O^{-1})^t|_{\Cinfty}\equiv \O.
     \end{equation*}
     Analogously, for every $f\in C^\infty(\T^n)$, we have that 
     \begin{equation*}
         \O^tf=\O^{-1}f\in\Cinfty\subset\D,
     \end{equation*}
     that is, 
      \begin{equation*}
         \O^t|_{C^\infty(\T^n)}\equiv \O^{-1}.
     \end{equation*}
 \end{prop}
  
\begin{proof}
    Let $f\in \Cinfty$. Then for $g\in C^\infty(\T^n)$, we have that
\begin{align*}
    \langle (\O^{-1})^tf,g\rangle&=\langle f,\O^{-1}g\rangle\\
    &=\frac{1}{(2\pi)^n}\int_{[0,2\pi]^n}\O f(x)\O \O^{-1}g(x)dx\\
    &=\frac{1}{(2\pi)^n}\int_{[0,2\pi]^n}\O f(x)g(x)dx\\
    &=\langle  \Omega_{\theta,T}f,g\rangle.
\end{align*}
Since this holds for every smooth periodic function $g$, the first claim in the statement follows. As for the second claim, the proof is analogous, since for $f\in C^\infty(\T^n)$ and $g\in \Cinfty$ we have that
\begin{align*}
    \langle \O^tf,g\rangle&=\langle f,\O g\rangle\\
    &=\frac{1}{(2\pi)^n}\int_{[0,2\pi]^n}f(x)\O g(x)dx\\
    &=\frac{1}{(2\pi)^n}\int_{[0,2\pi]^n}\O \O^{-1}f(x)\O g(x)dx\\
    &=\langle \O^{-1}f,g\rangle.
\end{align*}
\end{proof}

 We conclude this subsection by noting that for any given $u\in\D$, by the characterization of continuous linear functionals on Fréchet spaces, there exists $C=C_{u}>0$, $N=N_u\in\N_0$ such that
\begin{equation*}
    |\langle u,f\rangle|\leq C\tilde p_{N}(\O f)\lesssim Cp_N(f),
\end{equation*}
for every $f\in \Cinfty$.

\subsection{\texorpdfstring{A Fourier transform for $(\theta,T)$-periodic functions}{A Fourier transform for (θ,T)-periodic functions}}\label{sec_Fourier}

\begin{definition}\label{def_fourier}
    
    Given $f\in L^1_{\theta,T}(\R^n)$, define its Fourier coefficient at $\xi\in\Z^n$ by 
    \begin{align}\label{def_smooth_Fourier}
        \widehat{f}(\xi)\defeq\widehat{\O f}(\xi)&= \frac{1}{(2\pi)^n}\int_{[0,2\pi]^n} \O f(x)e^{-ix\cdot \xi}dx\notag\\
        &= \frac{1}{(2\pi)^n}\int_{[0,2\pi]^n}  \textstyle f(\frac{T}{2\pi}x)e^{-ix\cdot (\xi-i {\frac{\log(\theta)}{2\pi}})}dx,
    \end{align}
    where $\log(\theta)\defeq (\log(\theta_1),\dots,\log(\theta_n))$. 
    
    Similarly, for $u\in \D$ and $\xi\in\Z^n$, define its Fourier coefficient at $\xi$ by
    \begin{equation}\label{def_distrib_Fourier}
        \widehat{u}(\xi)\defeq \widehat{(\O^{-1})^{t}u}(\xi)=\langle (\O^{-1})^{t}u(x),e^{-ix\cdot \xi}\rangle=\langle u(x),e^{-i\frac{2\pi}{T}x\cdot (\xi-i{\frac{\log(\theta)}{2\pi}})}\rangle.
    \end{equation}
\end{definition}
 Since we can view $\Cinfty\subset  L^1_{\theta,T}(\R^n)$, the definition holds for every $f\in \Cinfty$.  Also note that for any $f\in \Cinfty$, we have that 
\begin{equation*}
    \widehat{u_f}(\xi)=\widehat{(\O^t)^{-1}f}(\xi)=\widehat{\O f}(\xi),
\end{equation*}
by Proposition \ref{prop_transpose_identity}, so definition \eqref{def_distrib_Fourier} is consistent  with \eqref{def_smooth_Fourier}.

Evidently the Fourier inversion formula implies that for $f\in \D$:
\begin{equation*}
    \O f(x)=\sum_{\xi\in\Z^n}\widehat{f}(\xi)e^{ix\cdot\xi},
\end{equation*}
and consequently
\begin{equation*}
    f(x)=\sum_{\xi\in\Z^n}\widehat{f}(\xi)e^{i\frac{2\pi}{T}x\cdot(\xi-i  { \frac{\log(\theta)}{2\pi}})},
\end{equation*}
in the sense of distributions, also with convergence in $L^{2}_{\theta,T}(\R^n)$ if $f\in L^{2}_{\theta,T}(\R^n)$ and also point-wise (uniformly in compact subsets of $\R^n$) if $f\in \Cinfty$.

\begin{obs}
    As already mentioned in the introduction of this paper, the Fourier transform/series defined above can be seen in the context of nonharmonic analysis. In fact, in this case the properties of the Fourier transform in Definition \ref{def_fourier} are closely tied to the study of completeness of sets of complex exponentials and also of biorthogonal systems and bases on Hilbert spaces. There exists an extensive literature on both of these subjects (though in the first one,  mostly restricted to dimension $n=1$). For a survey paper on completeness of sets of complex exponentials, we refer to \cite{Nonharmonic_survey}, and for a reference on biorthogonal systems and bases see \cite{Bari}.
\end{obs}

In the sequence, we present some properties of $(\theta, T)$-periodic Fourier series. 

First,  for $\theta\in\C^n_*$, set
    \begin{equation}\label{def_K_theta}
    K_{\theta,\min}\defeq  \prod_{|\theta_j|< 1}|\theta_j|\quad  \text{ and }\quad K_{\theta,\max}\defeq  \prod_{|\theta_j|> 1}|\theta_j|,
    \end{equation}
with the convention that the empty product is $1$. 

Note that 
\begin{equation*}
    \max_{x_j\in[0,1]} |\theta_j|^{x_j}=\begin{cases}
        |\theta_j|&\text{ if }|\theta_j|\geq 1,\\
        1,&\text{ if }|\theta_j|<1,
    \end{cases}
\end{equation*}
and 
\begin{equation*}
    \min_{x_j\in[0,1]} |\theta_j|^{x_j}=\begin{cases}
        1&\text{ if }|\theta_j|\geq 1,\\
        |\theta_j|&\text{ if }|\theta_j|<1.
    \end{cases}
\end{equation*}
Therefore 
\begin{equation*}
    K_{\theta,\min}=   \min_{x\in[0,1]^n}\prod_{j=1}^n |\theta_j|^{x_j}\quad\text{ and }\quad K_{\theta,\max}= \max_{x\in[0,1]^n} \prod_{j=1}^n |\theta_j|^{x_j}.
\end{equation*}

\begin{prop}\label{prop_Fourier_L1}
     For $f\in L_{\theta,T}^1(\R^n)$, we have that
     \begin{equation*}
         |\widehat{f}(\xi)|\leq \|f\|_{L^1_{\theta,T}(\R^n)}\leq  K_{\theta,\min}^{-1}\|f\|_{L^1([0,T]^n)}.
     \end{equation*}
\end{prop}
 \begin{proof}
         Indeed, the first inequality follows directly from Definition \ref{def_fourier}. As for the second inequality, note that
         \begin{align*}
            \|f\|_{L^1_{\theta,T}(\R^n)}&= \frac{1}{(2\pi)^n}\int_{[0,2\pi]^n}|f({\textstyle{\frac{2\pi}{T}}}x)|e^{-\sum_{j=1}^n\Real(\log(\theta_j))\frac{x_j}{2\pi}}dx\\
             &=\frac{1}{T^n}\int_{[0,T]^n}|f(x)|e^{-\log\Big(\prod_{i=1}^n|\theta_j|^{\frac{x_j}{T}}\Big)}dx\\
             &=\frac{1}{T^n}\int_{[0,T]^n}|f(x)|\Big(\prod_{i=1}^n|\theta_j|^{\frac{x_j}{T}}\Big)^{-1}dx\\
             &\leq  K_{\theta,\min}^{-1}\|f\|_{L^1([0,T]^n)}.
         \end{align*}
     \end{proof}

\begin{prop}\label{prop_L2}
    For $f\in L^2_{\theta,T}(\R^n)$, we have that
\begin{equation*}
   \left(\sum_{\xi\in\Z^n}|\widehat{f}(\xi)|^2\right)^{\frac{1}{2}}=\|f\|_{L^2_{\theta,T}(\R^n)},
\end{equation*}
and  
    \begin{equation} K_{\theta,\min}\left(\sum_{\xi\in\Z^n}|\widehat{f}(\xi)|^2\right)^{\frac{1}{2}}\leq  \|f\|_{L^2([0,T]^n)}\leq  K_{\theta,\max}\left(\sum_{\xi\in\Z^n}|\widehat{f}(\xi)|^2\right)^{\frac{1}{2}}.
    \end{equation}
    
\end{prop}
\begin{proof}
    From Plancherel's identity, we have that
\begin{equation*}
    \sum_{\xi\in\Z^n}|\widehat{f}(\xi)|^2=\|\O f\|_{L^2(\T^n)}^2=\|f\|_{L^2_{\theta,T}(\R^n)}^2,
\end{equation*}
for every $f\in L^2_{\theta,T}(\R^n)$, which proves the first claim in the statement. For the second claim, from the fact that
\begin{align}
    \sum_{\xi\in\Z^n}|\widehat{f}(\xi)|^2=\|f\|_{L^2_{\theta,T}(\R^n)}^2&=\frac{1}{T^n}\int_{[0,T]^n}|f(x)|^2\left(\prod_{j=1}^n|\theta_j|^{\frac{x_j}{T}}\right)^{-2}dx,\label{ineq_plancherel_proof}
\end{align}
it follows that
\begin{align*}
    \left(\prod_{|\theta_j|> 1}|\theta_j|\right)^{-2}\|f\|_{L^2([0,T]^n)}^2\le\sum_{\xi\in\Z^n}|\widehat{f}(\xi)|^2&\leq  \left(\prod_{|\theta_j|<1}|\theta_j|\right)^{-2}\|f\|_{L^2([0,T]^n)}^2,
\end{align*}
so that
\begin{equation*}
    K_{\theta,\min} \left(\sum_{\xi\in\Z^n}|\widehat{f}(\xi)|^2\right)^{\frac{1}{2}}\leq  \|f\|_{L^2([0,T]^n)}\leq  K_{\theta,\max} \left(\sum_{\xi\in\Z^n}|\widehat{f}(\xi)|^2\right)^{\frac{1}{2}},
\end{equation*}
as claimed.
\end{proof}

\begin{obs}
    Note that the existence of the constants $ K_{\theta,\min}$ and $K_{\theta,\max}$ for the  ``Plancherel inequality'' from Proposition \ref{prop_L2} is guaranteed from results in biorthogonal systems (see \cite{Bari}), as the pair of set of exponentials:
    \begin{equation*}
        \left\{e^{-i\frac{2\pi}{T}x\cdot(\xi-i\frac{\log(\theta)}{2\pi})}\right\}_{\xi\in\Z^n}\text{ and }\left\{e^{-i\frac{2\pi}{T}x\cdot(\xi+i\frac{\log(\theta)}{2\pi})}\right\}_{\xi\in\Z^n}
    \end{equation*}
    defines a biorthogonal system on $L^2([0,T]^n)$. Also note that in particular, if $\log(\theta)\in i\R^n$, then either one of these sets is orthogonal on $L^2([0,T]^n)$.
\end{obs}

\begin{prop}\label{prop_Fourier_properties}
    Let $f,g\in \Cinfty$. Then
    \begin{enumerate}
        \item[i)] $\widehat{f+\lambda g}(\xi)=\widehat{f}(\xi)+\lambda\widehat{g}(\xi)$,  $\lambda\in \C$;
        \item[ii)] $\widehat{\partial_{x_j}f}(\xi)=\frac{2\pi}{T}i\big(\xi_j-i\frac{\log(\theta_j)}{2\pi}\big)\widehat{f}(\xi)$, $1\leq j\leq n$;
        \item[iii)] $\widehat{e_{\xi_0}f(x)}(\xi)=\widehat{f}(\xi-\xi_0)$, where $e_{\xi_0}f(x)=e^{i\frac{2\pi}{T}x\cdot \xi_0}f(x)$, $\xi_0\in\Z^n$;
        \item[iv)] $\widehat{\tau_af}(\xi)=e^{i\frac{2\pi}{T}a\cdot (\xi-i\frac{\log(\theta)}{2\pi})}\widehat{f}(\xi)$, where $\tau_af(x)=f(x+a)$,  $a\in\R^n$;
        \item[v)] $\widehat{\mu_k{f}}(\xi)=\widehat{f}(\operatorname{sign}(k)\xi)$, where $\mu_k f(x)=f(k x)\in C^\infty_{\theta^{\operatorname{sign}(k)},\frac{T}{|k|}}(\R^n)$, $k\in\Z\backslash\{0\}$.
    \end{enumerate}
\end{prop}

\begin{proof}
{\it i)} Follows immediately from the linearity of $\O$ and the Fourier transform.

{\it ii)} Indeed, since the Fourier transform for smooth periodic functions converges uniformly, the Fourier series for smooth $(\theta,T)$-periodic functions converges uniformly over compact subset of $\R^n$. Consequently, for $1\leq j\leq n$ we have that
\begin{align*}
    \partial_{x_j}f(x)&=\partial_{x_j}\left[\sum_{\xi\in\Z^n}\widehat{f}(\xi)e^{i\frac{2\pi}{T}x\cdot(\xi-i {\frac{\log(\theta)}{2\pi}})}\right]\\
    &=\sum_{\xi\in\Z^n}\widehat{f}(\xi)\partial_{x_j}\left[e^{i\frac{2\pi}{T}x\cdot(\xi-i { \frac{\log(\theta)}{2\pi}})}\right]\\
    &=\sum_{\xi\in\Z^n}\frac{2\pi}{T}i\Big(\xi_j-i\frac{\log(\theta_j)}{2\pi}\Big)\widehat{f}(\xi)e^{i\frac{2\pi}{T}x\cdot(\xi-i { \frac{\log(\theta)}{2\pi}})},
\end{align*}
so by comparing Fourier coefficients we conclude that $\widehat{\partial_{x_j}f}(\xi)=\frac{2\pi}{T}i\big(\xi_j-i\frac{\log(\theta_j)}{2\pi}\big)\widehat{f}(\xi)$, as claimed.

{\it iii)} First note that $e_{\xi_0}f\in \Cinfty$. Thus, from Definition \ref{def_fourier}, we have that
\begin{align*}
      \widehat{e_{\xi_0}f}(\xi)&= \frac{1}{(2\pi)^n}\int_{[0,2\pi]^n}  \textstyle f(\frac{T}{2\pi}x)e^{ix\cdot \xi_0}e^{-ix\cdot (\xi-i{ \frac{\log(\theta)}{2\pi}})}dx\\
      &=\frac{1}{(2\pi)^n}\int_{[0,2\pi]^n}  \textstyle f(\frac{T}{2\pi}x)e^{-ix\cdot ((\xi-\xi_0)-i { \frac{\log(\theta)}{2\pi}})}dx\\
      &=\widehat{f}(\xi-\xi_0).
\end{align*}

{\it iv)}  Again note that $x\mapsto \tau_af(x)=f(x+a)\in \Cinfty$, hence
\begin{align*}
      \widehat{\tau_af}(\xi)&= \frac{1}{(2\pi)^n}\int_{[0,2\pi]^n}  {\textstyle f(\frac{T}{2\pi}x+a})e^{-ix\cdot (\xi-i { \frac{\log(\theta)}{2\pi}})}dx\\
      &=\frac{1}{(2\pi)^n}\int_{[0,2\pi]^n}  {\textstyle f(\frac{T}{2\pi}(x+\frac{2\pi}{T}a)})e^{-ix\cdot (\xi-i { \frac{\log(\theta)}{2\pi}})}dx\\
      &=\frac{1}{(2\pi)^n}\int_{\prod_{j=1}^n[\frac{2\pi}{T}a_j,2\pi+\frac{2\pi}{T}a_j]}  {\textstyle f(\frac{T}{2\pi}x)}e^{-i(x-\frac{2\pi}{T}a)\cdot (\xi-i { \frac{\log(\theta)}{2\pi}})}dx\\
      &=\widehat{f}(\xi)e^{i\frac{2\pi}{T}a\cdot (\xi-i\frac{\log(\theta)}{2\pi})}.
\end{align*}

    {\it v)} Recalling that $\mu_kf\in C^\infty_{\theta^{\operatorname{sign}(k)},\frac{T}{|k|}}(\R^n)$, by the definition of the Fourier transform, we have that
    \begin{align*}
        \widehat{\mu_kf}(\xi)&=\frac{1}{(2\pi)^n}\int_{[0,2\pi]^n} f\left({\textstyle\frac{\frac{T}{|k|}}{2\pi}}kx\right)e^{-ix\cdot(\xi-i { \frac{\log(\theta^{\operatorname{sign}(k)})}{2\pi}})}dx\\
        &=\frac{1}{(2\pi)^n}\int_{[0,2\pi]^n} f\left({\operatorname{sign}(k)\textstyle\frac{T}{2\pi}}x\right)e^{-ix\cdot(\xi-\operatorname{sign}(k)i { \frac{\log(\theta)}{2\pi}})}dx.
    \end{align*}
    If $\operatorname{sign}(k)=1$, from the equality above it is immediate that $\widehat{\mu_kf}(\xi)=\widehat{f}(\xi)$. Otherwise, $\operatorname{sign}(k)=-1$, so performing the change of variables $x\mapsto -x$ we obtain 
    \begin{align*}
        \widehat{\mu_kf}(\xi)
        &=\frac{1}{(2\pi)^n}\int_{[-2\pi,0]^n} f\left({\textstyle\frac{T}{2\pi}}x\right)e^{ix\cdot(\xi+i { \frac{\log(\theta)}{2\pi}})}dx\\
        &=\frac{1}{(2\pi)^n}\int_{[0,2\pi]^n} f\left({\textstyle\frac{T}{2\pi}}x\right)e^{ix\cdot(\xi+i { \frac{\log(\theta)}{2\pi}})}dx\\
        &=\frac{1}{(2\pi)^n}\int_{[0,2\pi]^n} f\left({\textstyle\frac{T}{2\pi}}x\right)e^{-ix\cdot(-\xi-i { \frac{\log(\theta)}{2\pi}})}dx\\
        &=\widehat{f}(-\xi),
    \end{align*}
    as claimed.
\end{proof}

\begin{obs}
   We note that the results in the Proposition  \ref{prop_Fourier_properties}  hold in larger function spaces, however we chose to state it as above for simplicity.
\end{obs}

\subsection{\texorpdfstring{$(\theta,T)$-periodic Sobolev spaces}{(θ,T)-periodic Sobolev spaces}}

Recall that the Sobolev spaces $H^s(\T^n)$ are given by the set of $u\in\mathcal{D}'(\T^n)$ such that
\begin{equation*}
    \|u\|^2_{H^s(\T^n)}=\|(\operatorname{Id}-\Delta_{\T^n})^{s/2}u\|^2_{L^2(\T^n)}=\sum_{\xi\in\Z^n}\langle\xi\rangle^{2s}|\widehat{u}(\xi)|^2<\infty,
\end{equation*}
where $s\in\R$ and $\langle\xi\rangle\defeq (1+\|\xi\|)^{2}$ denotes the Japanese bracket. 
\begin{definition}
    Given $s\in\R$, the $(\theta,T)$-periodic Sobolev space $H_{\theta,T}^s(\R^n)$ is defined as the set of all $u\in \D$ satisfying
    \begin{equation*}
        \|u\|^2_{H^s_{\theta,T}(\R^n)}\vcentcolon= \|\O u\|^2_{H^s(\T^n)}=\sum_{\xi\in\Z^n}\langle\xi\rangle^{2s}|\widehat{u}(\xi)|^2<\infty.
    \end{equation*}
\end{definition}

The mapping $H^s_{\theta,T}(\R^n)\ni u\mapsto \|u\|_{H^s_{\theta,T}(\R^n)}$ defines a norm, endowing $H^s_{\theta,T}(\R^n)$ with a Banach space topology. 

Evidently, $u\in H^s_{\theta,T}(\R^n)$ if and only if $\O u\in H^s(\T^n)$ and $H^0_{\theta,T}(\R^n)=L^2_{\theta,T}(\R^n)$. Therefore we immediately obtain that
\begin{equation*}
    \bigcup_{s\in\R}H^s_{\theta,T}(\R^n)=\mathcal{D}'(\R^n)\ \text{ and }\bigcap_{s\in\R}H^s_{\theta,T}(\R^n)=\Cinfty,
\end{equation*}
and consequently  we have the following results. 
\begin{prop}\label{prop_Fourier_coef}
    If $f\in \Cinfty$, then for every $N>0$, there exists $C_N>0$ such that
    \begin{equation}\label{ineq_rapid_decay}
        |\widehat{f}(\xi)|\leq C_N \langle\xi\rangle^{-N},
    \end{equation}
    for every $\xi\in\Z^n$. Conversely, given a sequence $\widehat{f}(\xi)$ satisfying \eqref{ineq_rapid_decay} we have that
    \begin{equation*}
\sum_{\xi\in\Z}\widehat{f}(\xi)e^{i\frac{2\pi}{T}x\cdot(\xi-i { \frac{\log(\theta)}{2\pi}})}\in \Cinfty.
    \end{equation*}
\end{prop}

\begin{prop}\label{prop_Fourier_coef_distrib}
    If $u\in \D$, then for every $N>0$, there exists $C_N>0$ such that
    \begin{equation}\label{ineq_rapid_decay_distrib}
        |\widehat{u}(\xi)|\leq C_N \langle\xi\rangle^{-N},
    \end{equation}
    for every $\xi\in\Z^n$. Conversely, given a sequence $\widehat{u}(\xi)$ satisfying \eqref{ineq_rapid_decay_distrib} we have that
    \begin{equation*}
\sum_{\xi\in\Z}\widehat{u}(\xi)e^{i\frac{2\pi}{T}x\cdot(\xi-i { \frac{\log(\theta)}{2\pi}})}\in \D.
    \end{equation*}
\end{prop}

We also have the following version of the Sobolev embedding theorem. 

\begin{prop}[Sobolev Embedding]\label{prop_Sobolev_embedding}
   Let $s>\frac{n}{2}$. Then for any $f\in H^s_{\theta,T}(\R^n)$, we have that
   \begin{align*}
       |f(x)|\leq       
       C\|f\|_{H^{s}_{\theta,T}(\R^n)}\prod_{j=1}^n|\theta_j|^{ \frac{x_j}{T}}\leq CK_{\theta,\max}\|f\|_{H^{s}_{\theta,T}(\R^n)}\prod_{j=1}^n|\theta_j|^{\lfloor \frac{x_j}{T}\rfloor},
   \end{align*}
   for almost every $x\in\R^n$ and some $C>0$ independent of $f$. Hence, $f$ is locally bounded and $H^{s}_{\theta,T}(\R^n)\subset L^\infty_{\theta,T}(\R^n)$.
\end{prop}
\begin{proof}
    Note that for $f\in H^s_{\theta,T}(\R^n)\subset L^2_{\theta,T}(\R^n)$,  from the Fourier inversion formula we have that
    \begin{align*}
        |f(x)|&\leq \sum_{\xi\in\Z^n}|\widehat{f}(\xi)|e^{\sum_{j=1}^n\Real(\log(\theta_j))\frac{x_j}{T}}\\
        &\leq \left(\sum_{\xi\in\Z^n}\langle \xi\rangle^{-2s}\right)^{\frac{1}{2}}\left(\sum_{\xi\in\Z^n}\langle \xi\rangle^{2s}|\widehat{f}(\xi)|^2\right)^{\frac{1}{2}}e^{\sum_{j=1}^n\log(|\theta_j|)\frac{x_j}{T}}\\
        &=C\left(\sum_{\xi\in\Z^n}\langle \xi\rangle^{2s}|\widehat{f}(\xi)|^2\right)^{\frac{1}{2}}\prod_{j=1}^n|\theta_j|^{\frac{x_j}{T}},
    \end{align*}
    for almost every $x\in\R^n$, where on the second line we applied the Cauchy-Schwartz inequality. This prove the first inequality. Using the fact that $a=\lfloor a\rfloor+\{a\}$, for any real number $a$, where $\{a\}\in [0,1)$ denotes the fractional part of $a$, we have that
\begin{align*}
    \prod_{j=1}^n|\theta_j|^{\frac{x_j}{T}}=\prod_{j=1}^n|\theta_j|^{\lfloor\frac{x_j}{T}\rfloor+\{\frac{x_j}{T}\}}\leq K_{\theta,\max}\prod_{j=1}^n|\theta_j|^{\lfloor\frac{x_j}{T}\rfloor},
\end{align*}
which, proves the second inequality.
\end{proof}

\section{Poincaré Inequality} \label{Section3}

The classical Poincaré inequality on the torus states that
\begin{equation*}
    \|\nabla f\|_{L^2(\T^n)}\geq \|f\|_{L^2(\T^n)},
\end{equation*}
for any  $f\in H^1(\T^n)$ with mean value zero. This can be restated as: for any $2\pi$-periodic $f:\R^n\to \C$ such that $f,\nabla f\in L^2([0,2\pi]^n)$ and $\widehat{f}(0)=0$, we have that
\begin{equation*}
    \|\nabla f\|_{L^2([0,2\pi]^n)}\geq \|f\|_{L^2([0,2\pi]^n)}.
\end{equation*}

We will now derive the analogue for $(\theta,T)$-periodic functions via the Fourier analysis developed in Section \ref{sec_Fourier}.

\begin{theorem}[Poincaré inequality for $(\theta,T)$-periodic functions]\label{theo_poincare}
     Let $T>0$ and $\theta\in \C^n_{*}$. For any $f\in H^1_{\theta,T}(\R^n)$ such that $\widehat{f}(\frac{i\log(\theta)}{2\pi})=0$ whenever $\frac{i\log(\theta)}{2\pi}\in\Z^n$, we have that
    \begin{equation}
        \|\nabla f\|_{L^2_{\theta,T}(\R^n)}\geq\begin{cases}
         \frac{2\pi}{T}\left\|\frac{i\log(\theta)}{2\pi}\right\|_{\C^n/\Z^n}\|f\|_{L^2_{\theta,T}(\R^n)}&\text{ if } \frac{i\log(\theta)}{2\pi}\not\in\Z^n,\\
            \frac{2\pi}{T}\|f\|_{L^2_{\theta,T}(\R^n)}&\text{ if }\frac{i\log(\theta)}{2\pi}=-\frac{\Arg(\theta)}{2\pi}\in\Z^n,
        \end{cases} 
    \end{equation}
    where  $\|\frac{i\log(\theta)}{2\pi}\|_{\C^n/\Z^n}$ denotes the distance from $\frac{i\log(\theta)}{2\pi}$ to $\Z^n\subset \C^n$.
\end{theorem}
\begin{proof}
    First note that the norm in $L_{\theta,T}^2(\R^n)$ is given by the inner product
    \begin{align*}
        \langle f,g\rangle_{L^2_{\theta,T}(\R^n)}=\langle \O f,\O g\rangle_{L^2([0,2\pi]^n)}&=\int_{[0,2\pi]^n}e^{-x\cdot {\frac{\log(\theta)}{2\pi}}}f(\textstyle\frac{T}{2\pi}x) \overline{e^{-x\cdot {\frac{\log(\theta)}{2\pi}}}g(\textstyle\frac{T}{2\pi}x)}dx\\
        &=\frac{1}{T^n}\int_{[0,T]^n}f(x)\overline{g(x)}e^{-2x\cdot { \frac{\Real(\log(\theta))}{T}}}dx\\
        &=\frac{1}{T^n}\int_{[0,T]^n}f(x)\overline{g(x)}e^{-2x\cdot { \frac{\log(|\theta|)}{T}}}dx.
    \end{align*}
  Note that the family of functions $\{ e^{i\frac{2\pi}{T}x\cdot(\xi-i {\frac{\log(\theta)}{2\pi}})}:\xi\in\Z^n\}$ is orthonormal with respect to this inner product, since for $\xi,\xi'\in 
    \Z^n$, we have that
    \begin{align*}
        \langle e^{i\frac{2\pi}{T}x\cdot(\xi-i {\frac{\log(\theta)}{2\pi}})},e^{i\frac{2\pi}{T}x\cdot(\xi'-i {\frac{\log(\theta)}{2\pi}})}\rangle_{L^2_{\theta,T}(\R^n)}&=\frac{1}{T^n}\int_{[0,T]^n}e^{-2x\cdot { \frac{\log(|\theta|)}{T}}}e^{i\frac{2\pi}{T}x\cdot(\xi-\xi')}e^{x\cdot { \frac{(\log(\theta)+\overline{\log(\theta)})}{T}}}dx\\
        &=\frac{1}{T^n}\int_{[0,T]^n}e^{-2x\cdot { \frac{\log(|\theta|)}{T}}}e^{i\frac{2\pi}{T}x\cdot(\xi-\xi')}e^{2x\cdot { \frac{\log(|\theta|)}{T}}}dx\\
        &=\frac{1}{T^n}\int_{[0,T]^n}e^{i\frac{2\pi}{T}x\cdot(\xi-\xi')}dx=\delta_{\xi,\xi'},
    \end{align*} 
    where $\delta_{\xi,\xi'}$ denotes the Kronecker delta. Therefore if  $\frac{i\log(\theta)}{2\pi}\not\in\Z^n$, we have that
    \begin{align*}
        \|\nabla f\|_{L^2_{\theta,T}(\R^n)}^2 &= {\langle \nabla f,\nabla f \rangle}_{L^2_{\theta,T}(\R^n)} \\
        &= \sum_{j=1}^n\langle \partial_{x_j}f(x),\partial_{x_j}f(x)\rangle_{L^2_{\theta,T}(\R^n)}\\
        &=\sum_{j=1}^n\Bigg\langle\sum_{\xi\in\Z^n} \Big(\frac{2\pi i}{T}\xi_j+\frac{\log(\theta_j)}{T}\Big)\widehat{f}(\xi)e^{i\frac{2\pi}{T}x\cdot(\xi-i { \frac{\log(\theta)}{2\pi}})},\\
        &\phantom{=}\phantom{\sum}\phantom{\Bigg\langle} \sum_{\xi'\in\Z^n} \Big(\frac{2\pi i}{T}\xi_j'+\frac{\log(\theta_j)}{T}\Big)\widehat{f}(\xi')e^{i\frac{2\pi}{T}x\cdot(\xi'-i { \frac{\log(\theta)}{2\pi}})}\Bigg\rangle_{L^2_{\theta,T}(\R^n)}\\
        &=\left(\frac{2\pi}{T}\right)^2\sum_{\xi\in\Z^n}\left\|i\xi+\frac{\log(\theta)}{2\pi}\right\|^2|\widehat{f}(\xi)|^2\\
        &\geq \left(\frac{2\pi}{T}\right)^2\left\|\frac{i\log(\theta)}{2\pi}\right\|_{\C^n/\Z^n}^2\sum_{\xi\in\Z^n}|\widehat{f}(\xi)|^2\\
        &\geq \left(\frac{2\pi}{T}\right)^{2}\left\|\frac{i\log(\theta)}{2\pi}\right\|_{\C^n/\Z^n}^2\|f\|_{L^2_{\theta,T}(\R^n)}^2,
    \end{align*}
    where on the last line we applied Proposition \ref{prop_L2}.
    On the other hand, if $\frac{i\log(\theta)}{2\pi}=-\frac{\Arg(\theta)}{2\pi}\in\Z^n$, then $|\theta_j|=1$, for every $j=1,\dots,n$ and by a similar argument
    \begin{align*}
         \|\nabla f\|_{L^2_{\theta,T}(\R^n)}^2&= \sum_{\xi\in\Z^n\backslash\{-\frac{\Arg(\theta)}{2\pi}\}}\left(\frac{2\pi}{T}\right)^2\left\|i\xi+\frac{\log(\theta)}{2\pi}\right\|^2|\widehat{f}(\xi)|^2\\
       & =\left(\frac{2\pi}{T}\right)^2 \sum_{\xi\in\Z^n\backslash\{-\frac{\Arg(\theta)}{2\pi}\}}\left\|\xi+\frac{\Arg(\theta)}{2\pi}\right\|^2|\widehat{f}(\xi)|^2\\
       &\geq \left(\frac{2\pi}{T}\right)^2\sum_{\xi\in\Z^n}|\widehat{f}(\xi)|^2\\
       &=\left(\frac{2\pi}{T}\right)^2\|f\|_{L^2_{\theta,T}(\R^n)}^2,
   \end{align*}
   where on the last line we applied Proposition \ref{prop_L2} once again, and also the fact that  $|\theta_j|=1$, for every $j=1,\dots,n$.
\end{proof}

Note that in Theorem \ref{theo_poincare}, the subspace where the inequality holds is given by the vanishing of possibly one Fourier coefficient.  This generalizes the periodic case, since the periodic case corresponds to  $\theta=(1,\dots,1)$, so $\frac{i\log(\theta)}{2\pi}=0\in\Z^n$ and  the subspace considered is given by all $f\in H^1$ satisfying $0=\widehat{f}(0)=\frac{1}{(2\pi)^n}\int_{[0,2\pi]^n}f(x)dx$.

\begin{corollary}
Let $T>0$ and $\theta\in \C_*^n$ with $|\theta_j|=1,\,j=1,\dots,n$. For any $f\in H^1_{\theta,T}(\R^n)$ such that $\widehat{f}(-\frac{\Arg(\theta)}{2\pi})=0$ whenever $\frac{\Arg(\theta)}{2\pi}\in\Z^n$, we have that
     \begin{equation}
        \|\nabla f\|_{L^2([0,T]^n)}\geq\begin{cases}
         \frac{2\pi}{T}\left\|\frac{\Arg(\theta)}{2\pi}\right\|_{\R^n/\Z^n}\|f\|_{L^2([0,T]^n)}&\text{ if } \frac{\Arg(\theta)}{2\pi}\not\in\Z^n,\\
            \frac{2\pi}{T}\|f\|_{L^2([0,T]^n)}&\text{ if }\frac{\Arg(\theta)}{2\pi}\in\Z^n.
        \end{cases} 
    \end{equation}
     where  $\|\frac{\Arg(\theta)}{2\pi}\|_{\C^n/\Z^n}$ denotes the distance from $\frac{\Arg(\theta)}{2\pi}$ to $\Z^n\subset \R^n$.
\end{corollary}

\begin{obs}
 Note that in the results above, one can verify that the constants $\|\frac{i\log(\theta)}{2\pi}\|_{\C^n/\Z^n}$ and $\|\frac{\Arg(\theta)}{2\pi}\|_{\R^n/\Z^n}$ are independent on the choice of complex logarithm/argument function.  
\end{obs}

\section{Global regularity of continuous linear operators}\label{Section4}

In this section, motivated by vast literature in the context of periodic functions, we study global regularity of differential operators in the sense of global hypoellipticity and global solvability.  We begin by proving a general result which connects these properties to the case of operators acting on periodic functions. First, we adapt the usual definitions of regularity to the $(\theta,T)$-periodic framework, as follows.

\begin{definition}
    Let $P:\Cinfty\to \Cinfty$ be a continuous linear operator. And consider by continuity $P:\D\to\D$. 
    We say that $P$ is globally hypoelliptic if for every $u\in \D$ such that $Pu=f\in\Cinfty$, we have that $u\in \Cinfty$.
\end{definition}

\begin{definition}
    Let $P:\Cinfty\to \Cinfty$ be a continuous linear operator.  Denote by  ${}^tP:\D\to \D$ the transpose of $P$, given by
    \begin{equation*}
        \langle {}^tPu,f\rangle =\langle u,Pf\rangle,
    \end{equation*}
    for every $u\in\D$ and $f\in \Cinfty$, and also let 
    \begin{equation*}
     (\ker {}^tP)^0=\Big\{f\in \Cinfty : \langle u,f\rangle =0\text{ for every } u\in\ker {}^tP\Big\}.
    \end{equation*}
    We say that $P$ is globally solvable if for every $f\in (\ker {}^tP)^0$, there exists $u\in \Cinfty$ such that $Pu=f$.
\end{definition}

Given a continuous linear operator $P:C^\infty(\R^n)\to C^\infty(\R^n)$, we will say that $P$ is globally hypoelliptic/solvable over $\Cinfty$ if $\Cinfty$ is invariant under $P$ (i.e.: $P(\Cinfty)\subset \Cinfty$)
and its restriction $P\equiv P|_{\Cinfty}:\Cinfty\to \Cinfty$ is globally hypoelliptic/solvable. By identifying $C^\infty(\T^n)$ with $C^\infty_{\bar{1},2\pi}(\R^n)$, we define $P$ to be globally hypoelliptic/solvable over $C^\infty(\T^n)$ in a similar way. Note that these coincide with the usual definitions on the periodic case.

\begin{lemma}\label{lemma_O_GH}
    Let $f\in \mathcal{D}'(\T^n)$, be such that  $\O^t f\in \Cinfty\subset \D$. Then $f\in C^\infty(\T^n)\subset \mathcal{D}'(\T^n)$, and consequently $\O^t f=\O^{-1}f$.
\end{lemma}
\begin{proof}
    Since $\O^t f\in \Cinfty$, by Proposition \ref{prop_transpose_identity} we have that
    \begin{equation*}
       f= (\O^{t})^{-1}\circ\O^t f=(\O^{-1})^{t}\circ\O^t f=\O \circ \O^t f\in C^\infty(\T^n).
    \end{equation*}
\end{proof}

\begin{theorem}\label{theo_gh_gs_general}
    Let $P:C^\infty(\R^n)\to C^\infty(\R^n)$ be a continuous linear operator such that  $\Cinfty$ is invariant under $P$. Then $P$ is globally hypoelliptic/solvable over $\Cinfty$ if and only if the operator $\tilde P=\O \circ P\circ \O^{-1}$ is globally hypoelliptic/solvable over $C^\infty(\T^n)$.
\end{theorem}
\begin{proof}
    First note that  by the definition of $\tilde P$,  the invariance of $\Cinfty$ under $P$ implies that $C^\infty(\T^n)$ is invariant under $\tilde P$. Moreover, we have that $\O \circ P=\tilde P\circ \O$,
        and since $(\O^{-1})^t|_{\Cinfty}=\O$ and $\O^t|_{C^\infty(\T^n)}=\O^{-1}$ as in Proposition \ref{prop_transpose_identity}, by continuity we also conclude that
       $(\O^{-1})^t\circ P=\tilde P\circ (\O^{-1})^t$ on $\D$.

    Next suppose that $\tilde P$ is globally hypoelliptic over $C^\infty(\T^n)$. Then if $Pu=f\in \Cinfty$ for some $u\in\D$, we have that
    \begin{align*}
         \tilde P\circ (\O^{-1})^t u = (\O^{-1})^t \circ P u=  (\O^{-1})^t  f=\O f\in C^\infty(\T^n),
    \end{align*}
    where the last equality follows from Proposition \ref{prop_transpose_identity}. We conclude that $(\O^{-1})^tu\in  C^\infty(\T^n)$ and consequently $u\in \Cinfty$ by Lemma \ref{lemma_O_GH}. Therefore $P$ is globally hypoelliptic over $\Cinfty$.

    Conversely, suppose that $P$ is globally hypoelliptic over $\Cinfty$. Note that by bijectivity of the respective mappings, we also have that 
    \begin{align}\label{eq_O_P}
        \O^{-1} \circ \tilde P=P\circ \O^{-1}
    \end{align}
    and
    \begin{equation*}
        P\circ \O^t=\O^t\circ \tilde P.
    \end{equation*}
     Hence, if $\tilde Pu=f\in C^\infty(\T^n)$ for some $u\in\mathcal{D}'(\T^n)$, we have that
    \begin{align*}
        P\circ \O^t u= \O^t\circ  \tilde Pu=\O^t f=\O^{-1}f\in \Cinfty.
    \end{align*}
Since $P$ is globally hypoelliptic over $\Cinfty$, we conclude that $\O^t u \in \Cinfty$ and thus $u\in C^\infty(\T^n)$ by Lemma \ref{lemma_O_GH}. Therefore $\tilde P$ is globally hypoelliptic over $C^\infty(\T^n)$.

Now assume that $\tilde P$ is globally solvable over $C^\infty(\T^n)$, and let $f\in (\ker {}^tP)^0$. 
Then note that $\O f\in (\ker {}^t\tilde P)^0$. Indeed, note that for any $u\in \ker {}^t\tilde P$, by \eqref{eq_O_P} we have that
 \begin{align*}
     \langle {}^tP\circ\O^t u,g\rangle&=\langle \O^tu,Pg\rangle=\langle u,\O\circ Pg\rangle\\
     &= \langle u,\tilde P\circ \O g\rangle=\langle {}^t \tilde Pu,\O g\rangle\\
     &=\langle0,\O g\rangle=0,
 \end{align*}
 for any $g\in \Cinfty$, and hence $\O u\in \ker {}^tP$. But then for any $u\in \ker {}^t\tilde P$, 
 \begin{equation*}
     \langle u,\O f\rangle =\langle \O^t u,f\rangle=0,
 \end{equation*}
 so $\O f\in (\ker {}^tP)^0$, as claimed. Then since $\tilde P$ is globally solvable over $C^\infty(\T^n)$, there exists $u\in C^\infty(\T^n)$ such that $\tilde Pu=\O f$, therefore
 \begin{align*}
     P\circ \O^{-1}u&=\O^{-1}\circ \tilde P u=\O^{-1}\circ \O f=f,
 \end{align*}
 and so $P$ is globally solvable over $\Cinfty$. The converse ($P$ globally solvable over $\Cinfty$ implies $\tilde P$ globally solvable over $C^\infty(\T^n)$) is proven  analogously.
\end{proof}

Next, we apply the previous results to study the explicit conditions for global hypoellipticity and global solvability of an important class of differential operators on the two-torus. More specifically, we consider vector fields with a zero order perturbation given by:
\begin{equation*}
    L=\partial_{x_1}+c(x_1)\partial_{x_2}+q(x_1,x_2),
\end{equation*}
where $c$ and $q$ are smooth complex-valued functions.

Initially, we consider operators with constant coefficients and using Theorem \ref{theo_gh_gs_general}, we give conditions for the operator to be globally hypoelliptic (resp. solvable).

To highlight important phenomena and improve the clarity of the proof, we consider three cases separately: constant coefficients, variable real-valued coefficients and variable complex-valued coefficients.

First we consider the case where the operator $L$ has constant coefficients, that is, $L:C^\infty(\R^2)\to C^\infty(\R^2)$ given by
    \begin{equation}\label{def_L_cte}
        L=\partial_{x_1}+c\partial_{x_2}+q
    \end{equation}
    where $c,q\in\mathbb{C}$. 

\begin{prop}
    \label{prop_GH_GS_cte}
    Let $\theta\in { \C^2_*}$ and $T>0$. The operator $L$ given by \eqref{def_L_cte} is globally hypoelliptic (resp. solvable) over $\cinfty$ if and only if one of the following equivalent conditions hold:
    \begin{enumerate}
        \item The operator
   \begin{equation*}
        \tilde L = \left[\partial_{x_1}+c\partial_{x_2}+\frac{\log(\theta_1)+c\log(\theta_2)}{2\pi}+q\frac{T}{2\pi}\right]
    \end{equation*}  
    is globally hypoelliptic (resp. solvable) over $C^\infty(\T^2)$.
    \item There exist $C,k>0$ such that
    \begin{equation*}
        \left|\xi_1+c\xi_2-i\left(\frac{\log(\theta_1)+c\log(\theta_2)}{2\pi}+q\frac{T}{2\pi}\right)\right|\geq C(1+\|\xi\|^2)^{-k},
    \end{equation*}
    for all but finitely many $\xi\in\Z^2$\\ (resp. for every $\xi\in\Z^2$ such that $\xi_1+c\xi_2-i\left(\frac{\log(\theta_1)+c\log(\theta_2)}{2\pi}+q\frac{T}{2\pi}\right)\neq 0$).
    \end{enumerate}
\end{prop}
\begin{proof}
For equivalence {\it (1)}, by Theorem \ref{theo_gh_gs_general} it is enough to show that 
\begin{equation*}
    {\frac{2\pi}{T} \tilde L}=\O \circ L\circ \O^{-1},
\end{equation*}
or equivalently, $\O^{-1}\circ \frac{2\pi}{T}\tilde L= L\circ \O^{-1}$.

Indeed, note that for $u\in C^\infty(\T^2)$, we have that 
\begin{align}
    \partial_{x_j}\circ \O^{-1} u(x)&=\partial_{x_j}\left[u({\textstyle\frac{2\pi}{T}x})e^{\log(\theta_1)\frac{x_1}{T}+\log(\theta_2)\frac{x_2}{T}}\right]\notag\\
    &=\frac{2\pi}{T}\partial_{x_j}u({\textstyle\frac{2\pi}{T}x})e^{\log(\theta_1)\frac{x_1}{T}+\log(\theta_2)\frac{x_2}{T}}+\frac{\log(\theta_j)}{T}u(\textstyle\frac{2\pi}{T}x)e^{\log(\theta_1)\frac{x_1}{T}+\log(\theta_2)\frac{x_2}{T}}\notag\\
    &=\frac{2\pi}{T}\O^{-1}\partial_{x_j}u(x)+\O^{-1}\left[\frac{\log(\theta_j)}{T}u(x)\right]\notag\\
    &=\frac{2\pi}{T}\O^{-1}\left[\partial_{x_j}u(x)+\frac{\log(\theta_j)}{2\pi}u(x)\right],\label{partial_effect_Omega}
\end{align}
    for every $x\in \T^2$ and $j=1,2$.  Hence
    \begin{align*}
        L\circ \O^{-1} u(x)&=\frac{2\pi}{T}\O^{-1}\left[(\partial_{x_1}+c\partial_{x_2})u(x)+\frac{\log(\theta_1)+c\log(\theta_2)}{2\pi}u(x)+q\frac{T}{2\pi} u(x)\right]\\
        &=\O^{-1}\circ \textstyle{\frac{2\pi}{T}} \tilde L u(x).
    \end{align*}
    In other words, $L\circ \O^{-1}=\O^{-1}\circ \frac{2\pi}{T}L_{\theta,T}$, as claimed. 
    
Finally, the proof of the equivalence  between {\it (1)} and {\it (2)} can be found in \cite{Berga_perturbations} and several other references, as it comes from the fact that the mapping $\xi\mapsto i\left(\xi_1+c\xi_2-i\left(\frac{\log(\theta_1)+c\log(\theta_2)}{2\pi}+q\frac{T}{2\pi}\right)\right)$ corresponds to the symbol of the operator $\tilde L$.
 \end{proof}

The previous proposition together with the results in \cite{Berga_perturbations}, yields the following corollary.

\begin{corollary}\label{coro_GH_cte}
    Let $L$ be the constant coefficients differential operator given by
    \begin{equation*}
        L=\partial_{x_1}+c\partial_{x_2}+q,
    \end{equation*}
    where $c,q\in\C$. If  $\Imag(c)\neq 0$, then $L$ is globally hypoelliptic and globally solvable over $\cinfty$. Moreover, if $q\in - \frac{(\log(\theta_1)+c\log(\theta_2)}{T}+\frac{2\pi i}{T} (\Z+c\Z)$, then $L$ is globally hypoelliptic over $\cinfty$ if and only if $\Imag (c)\ne 0$ or  $c\in \R$ is an irrational non-Liouville number, and globally solvable over $\cinfty$ if and only if  $\Imag (c)\ne 0$ or  $c\in \R$ is either rational or an irrational non-Liouville number.
\end{corollary}

   Note that if $q=0$, the condition $q\in - \frac{(\log(\theta_1)+c\log(\theta_2)}{T}+\frac{2\pi i}{T} (\Z+c\Z)$ becomes $\log(\theta_1)+c\log(\theta_2)\allowbreak \in {2\pi i} (\Z+c\Z)$, which is equivalent to $\theta_1,\theta_2=1$ and we recover the periodic case.

\begin{obs}\label{obs_indep_gh}
    Note that condition {\it (2)} in Proposition \ref{prop_GH_GS_cte} does not depend on which complex logarithm function was chosen: a different choice $\log'$ would satisfy $\log'(\theta_j)= \log(\theta_j)+2\pi ik$, $j=1,2$, for some $k\in\Z$, therefore
    \begin{equation*}
       \xi_1+c\xi_2-i\left(\frac{\log'(\theta_1)+c\log'(\theta_2)}{2\pi}+q\frac{T}{2\pi}\right)=(\xi_1+k)+c(\xi_2+k)-i\left(\frac{\log(\theta_1)+c\log(\theta_2)}{2\pi}+q\frac{T}{2\pi}\right),
    \end{equation*}
    hence note that a change of variables $\xi'=\xi+(k,k)\in\Z^2$ yields the exact same conditions. Similarly, it is clear that the conditions in Corollary \ref{coro_GH_cte}
    also do not depend on the choice of branch of logarithm considered. The same holds for the remaining  results in this section. 
\end{obs}

   Next we consider first order differential operators with variable coefficients. Since $\cinfty$ is a $C^\infty(\R^2/(T\Z)^2)$-module, we shall assume that all coefficients are $T$-periodic. This guarantees that $\cinfty$ is invariant under the action of these operators.

With this in mind, consider the differential operator $L:C^\infty(\R^2)\to C^\infty(\R^2)$ given by
    \begin{equation}\label{def_L_real}  L=\partial_{x_1}+a(x_1)\partial_{x_2}+q(x_1,x_2)
    \end{equation}
     where where $a\in C^\infty(\R/ T\Z,\R)$ and  $q\in C^\infty(\R^2/(T\Z)^2,\C)$. Also set $a_0=\frac{1}{T}\int_0^{T}a(x_1)dx_1$ and $q_0=\frac{1}{T^2}\int_0^{T}\int_0^Tq(x_1,x_2)dx_1dx_2$.
    \begin{prop}\label{prop_GH_GS_real}
         Let $L$ be the operator given by \eqref{def_L_real}.
         The operator $L$ is globally hypoelliptic/solvable over $\cinfty$ if and only if any of the following equivalent conditions hold:
         \begin{enumerate}
         \item  The operator \begin{equation*}
             \tilde L= \partial_{x_1}+a({\textstyle \frac{T}{2\pi} }x_1) \partial_{x_2} +\frac{\log(\theta_1)+a({\textstyle \frac{T}{2\pi} }x_1)\log(\theta_2)}{2\pi}+\frac{T}{2\pi}q({\textstyle \frac{T}{2\pi} }x_1,{\textstyle \frac{T}{2\pi} }x_2)
         \end{equation*}
      is globally hypoelliptic/solvable over $C^\infty(\T^2)$.
      \item The operator \begin{equation*}
             \tilde L_0= \partial_{x_1}+a_0 \partial_{x_2} +\frac{\log(\theta_1)+a({\textstyle \frac{T}{2\pi} }x_1)\log(\theta_2)}{2\pi}+\frac{T}{2\pi}q({\textstyle \frac{T}{2\pi} }x_1,{\textstyle \frac{T}{2\pi} }x_2)
         \end{equation*}
         is globally hypoelliptic/solvable over $C^\infty(\T^2)$.
          \end{enumerate}
        Moreover: if $q$ depends only on the first variable, or $a_0$ is an irrational non-Liouville number, any of the following equivalent conditions are also necessary and sufficient  for $L$ to be globally hypoelliptic/solvable over $\cinfty$: 
          \begin{enumerate}[(A)]
         \item The operator
  \begin{equation*}
             \tilde L_{00}= \partial_{x_1}+a_0 \partial_{x_2} +\frac{\log(\theta_1)+a_0\log(\theta_2)}{2\pi}+\frac{T}{2\pi}q_0
         \end{equation*}
           is globally hypoelliptic/solvable over $C^\infty(\T^2)$.
           \item  The operator 
    \begin{equation*}
        L_{00}=\partial_{x_1}+a_0\partial_{x_2}+q_0
    \end{equation*}
    is globally hypoelliptic/solvable over $\cinfty$.
           \item There exist $C,k>0$ such that
    \begin{equation*}
        \left|\xi_1+a_0\xi_2-i\left(\frac{\log(\theta_1)+a_0\log(\theta_2)}{2\pi}+\frac{T}{2\pi}q_0\right)\right|\geq C(1+\|\xi\|^2)^{-k},
    \end{equation*}
    for all but finitely many $\xi\in\Z^2$\\ (resp. for every $\xi\in\Z^2$ such that $\xi_1+a_0\xi_2-i\left(\frac{\log(\theta_1)+a_0\log(\theta_2)}{2\pi}+\frac{T}{2\pi}q_0\right)\neq 0$).
    \end{enumerate}
    \end{prop}
\begin{proof}
    First note that $\tilde L (C^\infty(\T^2))\subset C^\infty(\T^2)$, so as in the proof of Proposition \ref{prop_GH_GS_cte}, for equivalence {\it (1)} it is enough to show that $\O^{-1}\circ {\frac{2\pi}{T}}\tilde L=L\circ \O^{-1}$. Indeed, using \ref{partial_effect_Omega}, we have that 
    \begin{align*}
        L\circ \O^{-1} u(x)&=\O^{-1}\circ\partial_{x_1}u(x)+\O^{-1}\left[\frac{\log(\theta_1)}{T}u(x)\right]\\
        &\phantom{=}+c(x_1)\left(\O^{-1}\circ\partial_{x_2} u(x)+\O^{-1}\left[\frac{\log(\theta_2)}{T}u(x)\right]\right)+q(x_1,x_2)\cdot\O^{-1}u(x)\\
        &=\O^{-1}\left[\partial_{x_1}u(x)+\frac{\log(\theta_1)}{T}u(x)+a({\textstyle \frac{T}{2\pi}}x_1) \partial_{x_2} u(x)+a({\textstyle \frac{T}{2\pi}}x_1)\frac{\log(\theta_2)}{T}u(x)\right.\\
        &\phantom{=}\left.\vphantom{\frac{1}{T}}+q({\textstyle \frac{T}{2\pi}}x_1,{\textstyle\frac{T}{2\pi}}x_2)u(x) \right]\\
        &=\O^{-1}\circ {\textstyle\frac{2\pi}{T}}\tilde L.
    \end{align*}
    The equivalence between {\it (1)} and {\it (2)} follows from a classical conjugation induced by the mapping $\widehat{u}(t,\xi)\mapsto e^{i\int_0^ta(s)ds-a_0s}\widehat{u}(t,\xi)$, for every $(t,\xi)\in\T^1\times \Z$, $u\in\mathcal{D}'(\T^2)$, using that
    \begin{equation}\label{eq_a_0}
        \frac{1}{2\pi}\int_{0}^{2\pi}a({\textstyle\frac{T}{2\pi}}x_1)dx_1=\frac{1}{T}\int_{0}^{T}a(x_1)dx_1.
    \end{equation} The equivalence between {\it (1)} and {\it (A)} follows from results in \cite{Kow_thesis} and \cite{normal_form,vector_fields_perturbations}, with some minor adaptations, and from the analogous of formula \eqref{eq_a_0} for $q$. See also Propositions \ref{prop_appendix_1_variable} and \ref{prop_appendix_Berga} for more details. Equivalence between $(A)$, $(B)$ and $(C)$ follows from Proposition \ref{prop_GH_GS_cte}.
\end{proof}

Finally, we consider the case
     \begin{equation}\label{def_L_complex}
        L=\partial_{x_1}+c(x_1)\partial_{x_2}+q(x_1,x_2), 
    \end{equation}
    where both $c\in C^\infty(\R,\C)$, $c(x_1)=a(x_1)+ib(x_1)$, and  $q\in C^\infty(\R,\C)$ are $T$-periodic. Also set $a_0$ and $q_0$ as before, and $c_0=\frac{1}{T}\int_0^{T}c(x_1)dx_1$.
   
 \begin{prop}\label{prop_GH_GS_complex}
      The operator $L$ given by \eqref{def_L_complex}, is globally hypoelliptic/solvable over $\cinfty$ if and only if any of the following equivalent conditions hold: 
    \begin{enumerate}
        \item  The operator
   \begin{equation*}
        \tilde L = \partial_{x_1}+c({\textstyle\frac{T}{2\pi}}x_1)\partial_{x_2}+\frac{\log(\theta_1)+c({\textstyle\frac{T}{2\pi}}x_1)\log(\theta_2)}{2\pi}+\frac{T}{2\pi}q({\textstyle\frac{T}{2\pi}}x_1,{\textstyle\frac{T}{2\pi}}x_2)
    \end{equation*}
    is globally hypoelliptic/solvable over $C^\infty(\T^2)$.
    \item  The operator \begin{equation*}
        \tilde L_0 = \partial_{x_1}+(a_0+ib({\textstyle\frac{T}{2\pi}}x_1))\partial_{x_2}+\frac{\log(\theta_1)+c({\textstyle\frac{T}{2\pi}}x_1)\log(\theta_2)}{2\pi}+\frac{T}{2\pi}q({\textstyle\frac{T}{2\pi}}x_1,{\textstyle\frac{T}{2\pi}}x_2)
    \end{equation*}
    is globally hypoelliptic/solvable over $C^\infty(\T^2)$.
     \item  If $q$ depends only on the first variable, the operator \begin{equation*}
        \tilde L_{00} = \partial_{x_1}+(a_0+ib({\textstyle\frac{T}{2\pi}}x_1))\partial_{x_2}+\frac{\log(\theta_1)+c_0\log(\theta_2)}{2\pi}+{\frac{T}{2\pi}}q_0
    \end{equation*}
    is globally hypoelliptic/solvable over $C^\infty(\T^2)$.
    \end{enumerate}
    Moreover, if $b\not\equiv0$ and $b$ does not change sign, then the operator $L$ is globally hypoelliptic over $\cinfty$.
 \end{prop}
\begin{proof}
   The proof of equivalences {\it (1)} and {\it (2)} and {\it (3)}, follows analogous to the proof of Proposition \eqref{prop_GH_GS_real}. The last claim follows from {\it (1)} and \cite[Proposition 3.1]{Berga_perturbations}, or alternatively, from Proposition \ref{prop_appendix_Berga}.
\end{proof}

\appendix

\section{Global regularity for a class of first order differential operators on the torus}\label{sec_appendix}

The goal of this section is to present results on global regularity of a class of differential operators on the torus. These results have already been proven, but in more general or slightly distinct settings, hence we chose to include their statements and proofs for the sake of clarity and completeness.

The proof of this first proposition is an adaptation of the proof of \cite[Proposition 5.1.2]{Kow_thesis}. 

\begin{prop}\label{prop_appendix_1_variable}
    Consider the operator $\tilde L:C^\infty(\T^2)\to C^\infty(\T^2)$ given by
    \begin{equation*}
        \tilde L=\partial_{t}+ c(t)\partial_{x}+q(t),
    \end{equation*}
    where  $c,q\in C^\infty(\T^1)$, $c(t)=a(t)+ib(t)$. Then $L$ is globally hypoelliptic/solvable over $C^\infty(\T^2)$ if and only if 
    \begin{equation*}
        \tilde L_0=\partial_{t}+ (a_0+ib(t))\partial_{x}+q_0
    \end{equation*}
    is globally hypoelliptic/solvable over $C^\infty(\T^2)$, where 
    \begin{equation*}
      { a_0=\frac{1}{2\pi}\int_0^{2\pi}a(t)dt \,  \text{ and } \, q_0=\frac{1}{2\pi}\int_0^{2\pi}q(t)dt.}
    \end{equation*}
\end{prop}
\begin{proof}
    Define $\Psi:\mathcal{D}'(\T^2)\to \mathcal{D}'(\T^2)$ by
    \begin{equation}\label{def_Psi_appendix}
        \Psi u(t,x)=\sum_{\xi\in\Z}e^{Q(t)}\widehat{u}(t,\xi)e^{i\xi A(t)}e^{i\xi x},
    \end{equation}
    where
    \begin{equation*}
       { A(t)=\int_{0}^ta(\tau)d\tau-a_0t \, \text{ and } \, Q(t)=\int_0^tq(\tau)d\tau-q_0t,}
    \end{equation*}
    for every $t\in \T^1$. Then a simple computation shows that all the mappings above are well defined, $\Psi$ is bijective with  $\Psi(C^\infty(\T^2))=C^\infty(\T^2)$ and that $\Psi$ satisfies 
    \begin{equation*}
        \tilde L_0\circ \Psi=\Psi\circ \tilde L,
    \end{equation*}
    over $\mathcal{D}'(\T^2)$ and $C^\infty(\T^2)$. A standard argument allows us to conclude that $\tilde L$ is  globally hypoelliptic if and only if $\tilde L_0$ is globally hypoelliptic. Next, suppose that $\tilde L$ is globally solvable, and let $f\in (\ker {}^t\tilde L)^0$. Note that $e^{-2Q(t)}\Psi (\ker {}^t\tilde L)=(\ker {}^t\tilde L_0)$. Indeed, a simple computation shows that ${}^t\tilde L=-\tilde L+2q$ and ${}^t\tilde L_0=-\tilde L_0+2q_0$, hence if $u\in\ker {}^t\tilde L$ we have that
    \begin{align*}
        0&=e^{-2Q(t)}\Psi\circ {{}^t\tilde L} u(t,x)\\
         &=e^{-2Q(t)}\Psi\circ (-\tilde L+2q(t))u(t,x)\\
         &=e^{-2Q(t)}(-\tilde L_0+2q(t))\circ \Psi u(t,x)\\
         &=(-\tilde L_0+2q_0)\circ e^{-2Q(t)}\Psi u(t,x)\\
         &={}^t\tilde L_0(e^{-2Q(t)}\Psi u(t,x)),
    \end{align*}
    and so $e^{-2Q(t)}\Psi u\in \ker {}^t\tilde L_0$, which implies $e^{-2Q(t)}\Psi (\ker {}^t\tilde L)\subset(\ker {}^t\tilde L_0)$. Since the mapping $e^{-2Q(t)}\Psi:\mathcal{D}'(\T^2)\to\mathcal{D}'(\T^2)$ is bijective, with inverse: $u\mapsto\Psi^{-1} e^{2Q(t)}u$, we conclude that the equality $e^{-2Q(t)}\Psi (\ker {}^t\tilde L)=(\ker {}^t\tilde L_0)$ holds. Next, we claim that $\Psi^{-1} (\ker {}^t\tilde L_0)^0=(\ker {}^t\tilde L)^0$. Indeed, given $f\in (\ker {}^tL_0)^0$ and $u\in \ker {}^tL$, by what we have just shown we can write $u=\Psi^{-1} e^{2Q(t)}v$, for some $v\in (\ker {}^tL_0)$. Then
    \begin{align*}
        \langle  u,\Psi^{-1} f\rangle &= \langle  \Psi^{-1} e^{2Q(t)}v,\Psi^{-1} f\rangle\\
        &=\sum_{\xi\in\Z}\langle e^{-Q(t)}\widehat{v}(t,\xi)e^{-i\xi A(t)},e^{Q(t)}\widehat{f}(t,-\xi)e^{-i(-\xi)A(t)}\rangle\\
        &=\sum_{\xi\in\Z}\langle \widehat{v}(t,\xi),\widehat{f}(t,-\xi)\rangle\\
        &=\langle v,f\rangle=0.
    \end{align*}
Since $u\in \ker {}^tL$ was arbitrary, we have that  $\Psi^{-1} f\in (\ker {}^t\tilde L)^0$, and therefore we also conclude that $\Psi^{-1} (\ker {}^t\tilde L_0)^0=(\ker {}^t\tilde L)^0$. Again, the bijectivity of $\Psi$ allows us to conclude the equality of these sets.

Finally, suppose $f\in (\ker {}^t\tilde L_0)^0$. Then since  $\Psi^{-1} f\in (\ker {}^t\tilde L)^0$ and we are assuming $\tilde L$ globally solvable, there exists $u\in C^\infty(\T^2)$ such that $\tilde Lu = \Psi^{-1}f$. But then 
\begin{equation*}
    f=\Psi\circ \tilde Lu=\tilde L_0\circ \Psi u,
\end{equation*}
consequently, we conclude that $\tilde L_0$ is globally solvable. The converse follows analogously.
\end{proof}

The following proposition is inspired by some of the proofs in \cite{Berga_perturbations}.

\begin{prop}\label{prop_appendix_Berga}
    Consider the operator $\tilde L:C^\infty(\T^2)\to C^\infty(\T^2)$ given by
    \begin{equation*}
        \tilde L=\partial_{t}+ c(t)\partial_{x}+q(t,x),
    \end{equation*}
    where  $c,q\in C^\infty(\T^1)$, $c(t)=a(t)+ib(t)$. If $b\equiv0$ and $a_0=\frac{1}{2\pi}\int_0^{2\pi}a(t)dt$ is an irrational non-Liouville number, or $b\not\equiv0$ and $b$ does not change sign, then $\tilde L$ is globally hypoelliptic/solvable over $C^\infty(\T^2)$ if and only if 
    \begin{equation*}
        \tilde L_0=\partial_{t}+ (a_0+ib(t))\partial_{x}+q_0
    \end{equation*}
    is globally hypoelliptic/solvable over $C^\infty(\T^2)$, where $q_0=\frac{1}{(2\pi)^2}\int_0^{2\pi}q(t,x)dtdx$. Moreover, in the latter case, $ \tilde L_0$ is globally hypoelliptic and solvable (and therefore $\tilde L$ also).
\end{prop}
\begin{proof}
Note first that via a classical argument (see proof of Proposition \ref{prop_appendix_1_variable}) using the bijection  $\Psi_a:\mathcal{D}'(\T^2)\to\mathcal{D}'(\T^2)$ given by 
\begin{equation*}
    \Psi_au(t,x)=\sum_{\xi\in\Z}\widehat{u}(t,\xi)e^{i\xi(\int_0^ta(\tau)d\tau-a_0t)}e^{ix\xi},
\end{equation*}
we have that $\tilde L$ shares the same regularity properties as those of the operator
\begin{equation*}
    \tilde L_1= \partial_{t}+ (a_0+ib(t))\partial_{x}+q(t,x).
\end{equation*}
 Note that if $b\equiv0$ and $a_0=\frac{1}{2\pi}\int_0^{2\pi}a(t)dt$ is an irrational non-Liouville number,  then the operator $\tilde L'_0\vcentcolon=\tilde L_0-q_0=\partial_t+a_0\partial_x$ is globally hypoelliptic and globally solvable, by \cite{GW_Liouville,Hounie}. Let $\tilde q:\T^2\to \C$ be given by $\tilde q(t,x)=q(t,x)-q_0$. Then a simple computation shows that $\widehat{\tilde q}(\xi_1,\xi_2)|_{\xi=0}=q_0-q_0=0$, and since the symbol of $\tilde L_0'$ (which does not depend on $t$ nor on $x$) vanishes only for $\xi=0$, we conclude that $\tilde q\in (\ker {}^t\tilde L_0')^0$. Hence, there exists $\tilde Q\in C^\infty(\T^2)$ such that $\tilde L_0'\tilde Q=\tilde q$. Define  the multiplication operator $M_{\tilde Q}:\mathcal{D}'(\T^2)\to \mathcal{D}'(\T^2)$ given by $\mathcal{D}'(\T^2)\ni u(t,x)\mapsto e^{\tilde Q(t,x)}u(t,x)\in \mathcal{D}'(\T^2)$.

 Then for any $u\in\mathcal{D}'(\T^2)$, we have that
 \begin{align*}
     \tilde L_0\circ M_{\tilde Q}u&=(\tilde L_0-q_0)e^{\tilde Q(t,x)}u+q_0e^{\tilde Q(t,x)}u\\
     &=(q(t,x)-q_0)e^{\tilde Q(t,x)}u+e^{\tilde Q(t,x)}(\tilde L_0-q_0)u+q_0e^{\tilde Q(t,x)}u\\
     &=q(t,x)e^{\tilde Q(t,x)}u+e^{\tilde Q(t,x)}(\tilde L_0-q_0)u\\
     &=M_{\tilde Q}\circ \tilde L_1u,
 \end{align*}
so $ \tilde L_0\circ M_{\tilde Q} = M_{\tilde Q}\circ \tilde L_1$. Evidently, $M_{\tilde Q}$ is bijective and $M_{\tilde Q}(C^\infty(\T^2))=C^\infty(\T^2)$, hence by a standard argument we conclude that $\tilde L_0$ is globally hypoelliptic if and only if $\tilde L_1$ is globally hypoelliptic, and thus also if and only if $\tilde L$ is globally hypoelliptic. 

Next, note that $M_{\tilde Q}^{-1}(\ker {}^t\tilde  L_1)=\ker{}^t\tilde L_0$, for if $u\in \ker {}^t\tilde L_1$, then as in the proof of Proposition \ref{prop_appendix_1_variable}, we have that
\begin{align*}
    {}^t\tilde L_0\circ M_{\tilde Q}^{-1}u
    &=(-(\tilde L_0-q_0)+q_0)e^{-\tilde Q(t,x)}u\\
    &=(q(t,x)-q_0)e^{-\tilde Q(t,x)}u+e^{\tilde Q(t,x)}(-(\tilde L_0-q_0))u+e^{-\tilde Q(t,x)}q_0u\\
    &=e^{-\tilde Q(t,x)}(-(\tilde L_0-q_0)+q(t,x))u\\
    &=M_{\tilde Q}^{-1}\circ {}^t\tilde L_1u\\
    &=0,
\end{align*}
which together with the fact that $M_{\tilde Q}^{-1}$ is bijective, proves the claim. Then $M_{\tilde Q}(\ker {}^t\tilde L_1)^0=(\ker {}^t\tilde L_0)^0$, for if
$f\in (\ker {}^t\tilde L_1)^0$, and $u=M_{\tilde Q}^{-1}v\in \ker {}^t \tilde L_0$, where $v\in \ker {}^t\tilde L_1$ we have that
\begin{align*}
    \langle u,M_{\tilde Q}f\rangle&=\langle e^{-\tilde Q(t,x)}v(t,x),e^{\tilde Q(t,x)}f(t,x)\rangle\\
    &=\langle v(t,x),f(t,x)\rangle=0.
\end{align*}
This proves the inclusion $M_{\tilde Q}(\ker {}^t\tilde L_1)^0\subset(\ker {}^t\tilde L_0)^0$, while the reverse inclusion follows from bijectivity once again. Finally, assuming that $\tilde L_0$ is globally solvable and taking $f\in (\ker {}^t\tilde L_1)^0$, we have that $M_{\tilde Q}f\in (\ker {}^t\tilde L_0)^0$ and so there exists $u\in C^\infty(\T^2)$ such that $\tilde L_0u=f$. From this, we obtain that
\begin{equation*}
    f=M_{\tilde Q}^{-1}\circ \tilde L_0 u= \tilde L_1\circ M_{\tilde Q}^{-1} u,
\end{equation*}
and conclude that $\tilde L_1$ is also globally solvable. The proof of the converse assertion is analogous. 

Next consider the case where $b\not\equiv 0$ and $b $ does not change sign. Then by \cite{Hounie} the operator $\tilde L_0'=\tilde L_0-q_0=\tilde L_1-q(t,x)$ is globally hypoelliptic and globally solvable. Moreover, since $b_0=\frac{1}{2\pi}\int_0^{2\pi} b(t)dt\neq 0$, via the closed formula for the periodic solutions for the ordinary differential equations
\begin{equation*}
    \partial_t u(t)+\mu(t) u(t)=f(t),
\end{equation*}
where $\mu,f\in C^\infty(\T^1)$,
one can verify that $\ker \tilde L_0'\subset C^\infty(\T^2)$ is given by the set of constant functions (see Lemma \ref{lemma_solution_ode_var} for $(\theta,T)=(1,2\pi)$). Since $L_0'$ is a vector field,  ${}^t\tilde L_0'=-\tilde L_0'$, so its kernel also consists of constant functions (seen as distributions). Consequently, $(\ker {}^t\tilde L_0')^0$ is given by the set of smooth functions with whose Fourier coefficient at at $\xi=0$ is $0$. In other words, as in the case first part of this proof, it consists of the set of smooth functions with zero average. 
Again consider the function  $(t,x)\mapsto \tilde q(t,x)\vcentcolon= q(t,x)-q_0$, for every $(t,x)\in\T^2$. Then by the discussion above, $\tilde q$ is in $(\ker {}^t\tilde L_0')^0$.
Then rest of the proof of th equivalence is then analogous to the previous case ($b\equiv 0$ and $a_0$ irrational non-Liouville).

Next, we will prove that if $b\not\equiv 0$ and $b$ does not change sign, then $\tilde L_0$ is both globally hypoelliptic and globally solvable over $C^\infty(\T^2)$. Indeed, suppose that $\tilde L_0u=f\in C^\infty(\T^2)$. This implies that 
\begin{equation}\label{ODE_Fourier}
    \partial_t\widehat{u}(t,\xi)+(i \xi  (a_0+ib(t))+q_0)\widehat{u}(t,\xi)=\widehat{f}(t,\xi),
\end{equation}
for every $(t,\xi)\in\T^1\times \Z$.
For every $\xi\in\Z$ such that $i\xi c_0+q_0\not\in i\Z$, by Lemma \ref{lemma_solution_ode_var} we have that $\widehat{u}(\cdot,\xi)$ is given by
\begin{equation}\label{solution_minus}
        \widehat{u}(t,\xi)=(1-e^{2\pi\xi b_0-2\pi (i\xi a_0+q_0)})^{-1}\int_0^{2\pi}e^{ \xi \int_{t-s}^tb(\tau)d\tau}e^{-s(i \xi a_0+q_0)}\widehat{f}(t-s,\xi)ds,
    \end{equation}
or equivalently, by
\begin{equation}\label{solution_plus}
     \widehat{u}(t,\xi)=(e^{-2\pi \xi b_0+2\pi (i\xi c_0+q_0)}-1)^{-1}\int_0^{2\pi}e^{ -\xi \int_{t}^{t+s}b(\tau)d\tau}e^{(i \xi a_0+q_0)s}\widehat{f}(t+s,\xi)ds,
\end{equation}
for every $t\in \T^1$.
On the other hand, for $\xi\in\Z$ such that $i\xi c_0+q_0\in i\Z$, $\widehat{u}(\cdot,\xi)$ is given by 
\begin{equation}
    \widehat{u}(t,\xi)=\lambda e^{-i\xi\int_0^ta_0+ib(\tau)d\tau-q_0t}+\int_0^t e^{-i\xi \int_s^t a_0+ib(\tau)d\tau+(s-t)q_0} \widehat{f}(s,\xi)ds,
\end{equation}
for some $\lambda\in\C$. But note that $i\xi c_0+q_0\in i\Z$ for only finitely many $\xi\in\Z$, since
\begin{equation*}
  \operatorname{Re}(i\xi c_0+q_0)=\operatorname{Re}(q_0)-b_0\xi
\end{equation*}
is zero for at most one value of $\xi$, since the fact that $b$ does not change sign implies that $b_0\neq 0$. Hence to determine the decay of the partial  Fourier coefficients $\widehat{u}(\cdot,\xi)$, it is enough to consider the case $i\xi c_0+q_0\not\in i\Z$.

 First assume that $b_0<0$. Then for $\xi>0$ sufficiently large, we have that 
 \begin{equation*}
     |e^{2\pi\xi b_0-2\pi (i\xi a_0+q_0)}|<1/2,
 \end{equation*}
 so the triangle inequality applied to \eqref{solution_minus} implies that
     \begin{align*}
         |\widehat{u}(t,\xi)|&\leq |1-e^{2\pi\xi b_0-2\pi (i\xi a_0+q_0)}|^{-1}\int_0^{2\pi} |e^{ \xi \int_{t-s}^tb(\tau)d\tau}e^{-s\operatorname{Re}(q_0)}\widehat{f}(t-s,\xi)|ds\\
         &\leq 2\max_{0\leq s\leq 2\pi }e^{-s\operatorname{Re}(q_0)}\max_{s\in \T^1}|\widehat{f}(s,\xi)|\int_0^{2\pi}ds\\
         &\leq C_{1} \max_{s\in \T^1}|\widehat{f}(s,\xi)|,
     \end{align*}
     for every $t\in \T^1$ and some $C_1>0$. For $\xi<0$ sufficiently large, a similar argument applied to \eqref{solution_plus} yields
     \begin{equation*}
         |\widehat{u}(t,\xi)|\leq C_2\max_{s\in\T^1}|\widehat{f}(s,\xi)|,
     \end{equation*}
     for every $t\in \T^1$ and some $C_2>0$. Using the fact that the previous estimates for $\widehat{u}(\cdot,\xi)$ may not hold for a finite number of $\xi\in\Z$, we conclude that there exists $C_3>0$ such that
 \begin{equation*}
         |\widehat{u}(t,\xi)|\leq C_3\max_{s\in\T^1}|\widehat{f}(s,\xi)|,
     \end{equation*}
      for every $(t,\xi)\in \T^1\times \Z$.
     Since $f$ is smooth, for every $N>0$ there exists $C'_N>0$ such that
      \begin{equation}\label{estimate_Fourier_u}
         |\widehat{u}(t,\xi)|\leq C_3C'_N\langle \xi \rangle^{-N},
     \end{equation}
     for every $\xi\in\Z$. Finally, note that \eqref{ODE_Fourier} implies that
     \begin{equation*}
         |\partial_t^{n}\widehat{u}(t,\xi)|\leq |\partial_t^{n-1}\widehat{f}(t,\xi)|+|\xi|\sum_{j=0}^{n-1}\binom{n-1}{j}|\partial_t^{j}[a_0+ib(t)]||\partial_t^{n-1-j}\widehat{u}(t,\xi)|+|q_0||\partial_t^{n-1}\widehat{u}(t,\xi)|,
     \end{equation*}
     for every $t\in\T^1$ and $\xi\in\Z$. Then an inductive argument, together with \eqref{estimate_Fourier_u} implies that for ever $n\in\N$ and $N>0$ there exists $C_{nN}$ such that
     \begin{equation*}
         |\partial_t^{n}\widehat{u}(t,\xi)|\leq C_{nN}\langle \xi \rangle^{-N},
     \end{equation*}
     for every $(t,\xi)\in\T^1\times \Z$. Therefore $u\in C^\infty(\T^2)$ and so $\tilde L_0$ is globally hypoelliptic over $C^\infty(\T^2)$. The fact that $\tilde L_0$ is globally solvable then follows from the fact that global hypoellipticity implies global solvability (see \cite[Theorem 3.5]{Araujo_sumsofsquares}).
\end{proof}
\section{\texorpdfstring{$(\theta,T)$-periodic solutions to a class of first order differential equations}{(θ,T)-periodic solutions to a class of first order differential equations}}\label{sec_app_solution}
\begin{lemma}\label{lemma_solution_ode_cte}
        Let $\theta\in\C\backslash\{0\},\,T>0$ and  $\lambda\in\mathbb{C}.$ If $\theta\neq e^{-T\lambda}$, then the equation
        \begin{equation}\label{eq-ode}
            \partial_xu(x)+\lambda u(x)=f(x),
        \end{equation}
        where $f\in C^\infty_{\theta,T}(\R)$ admits unique $(\theta,T)$-periodic solution $u\in C^\infty_{\theta,T}(\R)$ given by
        \begin{equation}\label{eq-u-formula1}
            u(x)=\frac{1}{1-\theta^{-1}e^{-T\lambda}}\int_0^Te^{-\lambda s}f(x-s)ds,
        \end{equation}
        or equivalently by
        \begin{equation}\label{eq-u-formula2}
            u(x)=\frac{1}{\theta e^{T\lambda}-1}\int_0^Te^{\lambda s}f(x+s)ds,
        \end{equation}
        for every $x\in\R$. On the other hand, if $\theta= e^{-T\lambda}$, then \eqref{eq-ode} admits (infinitely many) $(\theta,T)$-periodic solutions $u_c\in C^\infty_{\theta,T}(\R)$ given by 
        \begin{equation}\label{eq_u_formula_theta}
            u_c(x)= ce^{-\lambda x}+ e^{-\lambda x}\int_0^x e^{\lambda s}f(s)ds,
        \end{equation}
        for every $x\in\R$ and $c\in\mathbb{C}$ if and only if 
        \begin{equation}\label{compatibility_1}
         \int_0^Te^{\lambda s}f(s)ds=0.
     \end{equation}
    \end{lemma}
    \begin{proof}
        First assume that $\theta\neq e^{-T\lambda}$. We verify directly that $u$ given by \eqref{eq-u-formula1} satisfies \eqref{eq-ode}. Indeed, note that
        \begin{align*}
            \partial_xu(x)&=-\frac{1}{1-\theta^{-1}e^{-T\lambda}}\int_0^Te^{-\lambda s}\partial_sf(x-s)ds\\
            &=-\frac{1}{1-\theta^{-1}e^{-T\lambda}}\left\{\left[e^{-T\lambda}f(x-2\pi)-f(x)\right]+\lambda\int_0^Te^{-\lambda s}f(x-s)ds\right\}\\
            &=-\frac{1}{1-\theta^{-1}e^{-T\lambda}}\left\{(\theta^{-1}e^{-T\lambda}-1)f(x)\right\}-\lambda u(x)\\
            &=f(x)-\lambda u(x),
        \end{align*}
        for every $x\in\R$. Also notice that since $f$ is $(\theta,T)$-periodic, it follows that $u$ is also $(\theta,T)$-periodic and clearly $\partial_xu$ is continuous, since both $u$ and $f$ are continuous. The proof that \eqref{eq-u-formula2} also defines a $(\theta,T)$-periodic solution is similar. 
        
        To verify that both formulas are equivalent, notice that one can obtain \eqref{eq-u-formula2} from \eqref{eq-u-formula1} performing the change of variables $ s\mapsto -s'+T$.

        To see that such solution is unique, suppose that $u_1$ and $u_2$ are both smooth $(\theta,T)$-periodic  solutions to \eqref{eq-ode}. Then $u=u_1-u_2$ defines a smooth $(\theta,T)$-periodic  solutions to the differential equation
        \begin{equation*}
            \partial_xu(x)+\lambda u(x)=0.
        \end{equation*}
        Since $\theta\neq e^{-T\lambda}$, it is easy to verify that the $(\theta,T)$-periodic solution to this ODE is $u\equiv0$ and therefore $u_1=u_2$, proving the uniqueness claim.
    
    Next, assume that $\theta= e^{-T\lambda}$. Note that differentiating \eqref{eq_u_formula_theta} we obtain:
    \begin{align*}
        \partial_xu(x)&=-\lambda C e^{-\lambda x} -\lambda  e^{-\lambda x}\int_0^x e^{\lambda s}f(s)ds+e^{-\lambda x}e^{\lambda x}f(x)\\
        &=-\lambda \left( C e^{-\lambda x}+e^{-\lambda x}\int_0^x e^{\lambda s}f(s)ds\right)+f(x)\\
        &=-\lambda u(x)+f(x),
    \end{align*}
    therefore $u$ given by \eqref{eq_u_formula_theta} solves \eqref{eq-ode}. Also, 
    \begin{align*}
        u(x+T)&=Ce^{-\lambda x}e^{-T\lambda}+ e^{-\lambda x}e^{-T\lambda}\int_0^{x+T} e^{\lambda s}f(s)ds\\
        &=e^{-T\lambda}\left(Ce^{-\lambda x}+e^{-\lambda x}\int_0^Te^{\lambda s}f(s)ds+e^{-\lambda x}\int_T^{x+T}e^{\lambda s}f(s)ds\right)\\
        &=\theta\left(Ce^{-\lambda x}+e^{-\lambda x}\int_0^{x}e^{\lambda( s'+T)}f(s'+T)ds'\right)\\
        &=\theta\left(Ce^{-\lambda x}+e^{-\lambda x}\int_0^{x}e^{\lambda s'}\theta ^{-1}\theta f(s')ds'\right)+\theta e^{-\lambda x}\int_0^Te^{\lambda s}f(s)ds\\   
        &=\theta u(x)+\theta e^{-\lambda x}\int_0^Te^{\lambda s}f(s)ds,
    \end{align*}
    for every $x\in\R$, where on the second to last line we used the fact that $e^{\lambda T}=\theta^{-1}$ and that $f$ is $(\theta,T)$-periodic. Hence, we conclude that $u$ is $(\theta,T)$-periodic if and only if $\int_0^Te^{\lambda s}f(s)ds=0$, as claimed. The proof that these are the only possible solutions follows from the simple verification that the only smooth $(\theta,T)$-periodic solutions to $\partial_t u(x)+\lambda u(x)=0$ are given by $u_c(x)=ce^{-\lambda x}$.
\end{proof}
\begin{lemma}\label{lemma_solution_ode_var}
     Let $\theta\in\C\backslash\{0\},\,T>0$ and  $\lambda\in C^\infty_{1,T}(\R)\sim C^\infty(\R/T\Z).$ Set $\lambda_0\vcentcolon=\frac{1}{T}\int_0^T\lambda(x)dx$. If $\theta\neq e^{-T\lambda_0}$, then the equation
        \begin{equation}\label{eq_ode_2}
            \partial_xu(x)+\lambda(x) u(x)=f(x),
        \end{equation}
        where $f\in C^\infty_{\theta,T}(\R)$ admits unique $(\theta,T)$-periodic solution $u\in C^\infty_{\theta,T}(\R)$ given by
        \begin{equation}\label{eq_u_formula1_2}
            u(x)=\frac{1}{1-\theta^{-1}e^{-T\lambda_0}}\int_0^Te^{-\int_{x-s}^x\lambda (\tau)d\tau}f(x-s)ds,
        \end{equation}
        or equivalently by
        \begin{equation}\label{eq_u_formula2_2}
            u(x)=\frac{1}{\theta e^{T\lambda_0}-1}\int_0^Te^{\int_{x}^{x+s}\lambda (\tau)d\tau}f(x+s)ds,
        \end{equation}
        for every $x\in\R$. On the other hand, if $\theta= e^{-T\lambda_0}$, then \eqref{eq_ode_2} admits (infinitely many) $(\theta,T)$-periodic solutions $u_c\in C^\infty_{\theta,T}(\R)$ given by 
        \begin{equation}\label{eq_u_formula_theta_2}
            u_c(x)= ce^{-\int_0^x\lambda(s)ds}+ \int_0^x e^{-\int_s^x\lambda(\tau)d \tau}f(s)ds,
        \end{equation}
        for every $x\in\R$ and $c\in\mathbb{C}$ if and only if 
        \begin{equation}\label{compatibility_2}
         \int_0^Te^{\int_0^s\lambda(\tau)d\tau}f(s)ds=0.
         \end{equation}
         \begin{proof}
             We will show that in both cases $(\theta=e^{-T\lambda_0}$ and $\theta\neq e^{-T\lambda_0}$), a smooth $(\theta,T)$-periodic function $u(x)$ is a solution to \eqref{eq-ode}, with $\lambda$ replaced by $\lambda_0$, if and only if $v(x)=e^{\int_0^x\lambda(\tau)d\tau-\lambda_0x}u(x)$ is a solution to 
             \begin{equation}\label{eq_ode_aux}
                 \partial_x v(x)+\lambda_0v(x)=e^{\int_0^x\lambda(\tau)d\tau-\lambda_0x}f(x).
             \end{equation}
             This proves the statement, since then if $\theta\neq e^{-T\lambda_0}$ by Lemma \ref{lemma_solution_ode_cte} we conclude that the unique $(\theta,T)$-periodic solution to \eqref{eq_ode_2} is given by
             \begin{align*}
                 e^{\int_0^x\lambda(\tau)d\tau-\lambda_0x}u(x)=\frac{1}{1-\theta^{-1}e^{-T\lambda_0}}\int_0^Te^{-\lambda_0 s}e^{\int_0^{x-s}\lambda(\tau)d\tau-\lambda_0(x-s)}f(x-s)ds,
             \end{align*}
             which can be rewritten as
             \begin{equation*}
                 u(x)=\frac{1}{1-\theta^{-1}e^{-T\lambda_0}}\int_0^Te^{-\int_{x-s}^x\lambda(\tau)d\tau}f(x-s)ds,
             \end{equation*}
             for every $x\in\R$, and similarly for the other equivalent solution, proving the case $\theta\neq e^{-T\lambda_0}$. For the case $\theta\neq e^{-T\lambda_0}$, first note that
             \begin{equation*}
                 \int_0^Te^{\int_0^s\lambda(\tau)d\tau}f(s)ds\iff \int_0^Te^{\lambda_0s}e^{\int_0^s\lambda(\tau)d\tau-\lambda_0s}f(s)ds=0,
             \end{equation*}
             so $f$ satisfies \eqref{compatibility_2} if and only if $e^{\int_0^x\lambda(\tau)d\tau-\lambda_0x}f(x)$ satisfies the compatibility condition \eqref{compatibility_1} with $\lambda$ replaced by $\lambda_0$. Therefore if the condition \eqref{compatibility_2} is met, by Lemma \ref{lemma_solution_ode_cte} we have that the solutions of \eqref{eq_ode_2} are given by
             \begin{equation*}
                  e^{\int_0^x\lambda(\tau)d\tau-\lambda_0x}u_c(x)=ce^{-\lambda_0 x}+ e^{-\lambda_0 x}\int_0^x e^{\lambda_0 s}e^{\int_0^s\lambda(\tau)d\tau-\lambda_0s}f(s)ds,
             \end{equation*}
             or equivalently
             \begin{equation*}
                 u_c(x)=ce^{-\int_0^x\lambda(s)ds}+ \int_0^x e^{-\int_s^x\lambda(\tau)d \tau}f(s)ds
             \end{equation*}
             for every $x\in\R$ and $c\in\C$.

             Therefore, all that is left is to prove the first claim in this proof. To see why it is true, first assume that $u$ solves \eqref{eq_ode_2}. Then
             \begin{align*}
                 \partial_xv(x)&=(\lambda (x)-\lambda_0)e^{\int_0^x\lambda(\tau)d\tau-\lambda_0x}u(x)+e^{\int_0^x\lambda(\tau)d\tau-\lambda_0x}\partial_x u(x)\\
                 &=-\lambda_0v(x)+e^{\int_0^x\lambda(\tau)d\tau-\lambda_0x}(\partial_x u(x)+\lambda (x)u(x))\\
                 &=-\lambda_0v(x)+e^{\int_0^x\lambda(\tau)d\tau-\lambda_0x}f(x),
             \end{align*}
             for every $x\in\R$, so that $v(x)=e^{\int_0^x\lambda(\tau)d\tau-\lambda_0x}u(x)$ solves \eqref{eq_ode_aux}. Conversely, suppose that $v(x)=e^{\int_0^x\lambda(\tau)d\tau-\lambda_0x}u(x)$ solves \eqref{eq_ode_aux}. Then 
             \begin{align*}
                  \partial_xu(x)&=\partial_x[e^{-\int_0^x\lambda(\tau)d\tau+\lambda_0x}v(x)]\\
                  &=(\lambda_0-\lambda(x))e^{-\int_0^x\lambda(\tau)d\tau+\lambda_0x}v(x)+e^{-\int_0^x\lambda(\tau)d\tau+\lambda_0x}\partial_xv(x)\\
                  &=-\lambda(x)u(x)+e^{-\int_0^x\lambda(\tau)d\tau+\lambda_0x}(\partial_x v(x)+\lambda_0 v(x))\\
                  &=-\lambda(x)u(x)+e^{-\int_0^x\lambda(\tau)d\tau+\lambda_0x}e^{\int_0^x\lambda(\tau)d\tau-\lambda_0x}f(x)\\
                  &=-\lambda(x)u(x)+f(x),
             \end{align*}
             for every $x\in\R$, so that $u$ solves \eqref{eq_ode_2}. 
         \end{proof}
\end{lemma}
\bibliographystyle{plain}
\bibliography{references}

\end{document}